\documentclass[a4paper,12pt]{amsart}
\title[Bogomolov-Gieseker type inequality]{{\bf Bogomolov-Gieseker type inequality and 
counting invariants}}
\date{}
\author{Yukinobu Toda}

\usepackage{amscd}
\usepackage{amsmath}
\usepackage{amssymb}
\usepackage{amsthm}
\usepackage{float}
\usepackage[dvips]{graphicx}

\usepackage[all,ps,dvips]{xy}

\usepackage{array}
\usepackage{amscd}
\usepackage[all]{xy}

\DeclareFontFamily{U}{rsfs}{%
\skewchar\font127}
\DeclareFontShape{U}{rsfs}{m}{n}{%
<-6>rsfs5<6-8.5>rsfs7<8.5->rsfs10}{}
\DeclareSymbolFont{rsfs}{U}{rsfs}{m}{n}
\DeclareSymbolFontAlphabet
{\mathrsfs}{rsfs}
\DeclareRobustCommand*\rsfs{%
\@fontswitch\relax\mathrsfs}

\theoremstyle{plain}
\newtheorem{thm}{Theorem}[section]
\newtheorem{prop}[thm]{Proposition}
\newtheorem{lem}[thm]{Lemma}

\newtheorem{defi}[thm]{Definition}
\newtheorem{rmk}[thm]{Remark}
\newtheorem{cor}[thm]{Corollary}

\newtheorem{prop-defi}[thm]{Proposition-Definition}
\newtheorem{thm-defi}[thm]{Theorem-Definition}
\newtheorem{lem-defi}[thm]{Lemma-Definition}

\newtheorem{conj}[thm]{Conjecture}
\newtheorem{exam}[thm]{Example}

\newdimen\argwidth
\def\db[#1\db]{
 \setbox0=\hbox{$#1$}\argwidth=\wd0
 \setbox0=\hbox{$\left[\box0\right]$}
  \advance\argwidth by -\wd0
 \left[\kern.3\argwidth\box0 \kern.3\argwidth\right]}

\newcommand{\aA}{\mathcal{A}}
\newcommand{\bB}{\mathcal{B}}

\newcommand{\eE}{\mathcal{E}}
\newcommand{\fF}{\mathcal{F}}

\newcommand{\hH}{\mathcal{H}}

\newcommand{\lL}{\mathcal{L}}
\newcommand{\mM}{\mathcal{M}}

\newcommand{\oO}{\mathcal{O}}

\newcommand{\sS}{\mathcal{S}}
\newcommand{\tT}{\mathcal{T}}
\newcommand{\uU}{\mathcal{U}}

\newcommand{\xX}{\mathcal{X}}
\newcommand{\yY}{\mathcal{Y}}
\newcommand{\zZ}{\mathcal{Z}}

\newcommand{\Supp}{\mathop{\rm Supp}\nolimits}
\newcommand{\Hom}{\mathop{\rm Hom}\nolimits}

\newcommand{\dR}{\mathbf{R}}

\newcommand{\Hilb}{\mathop{\rm Hilb}\nolimits}

\newcommand{\Pic}{\mathop{\rm Pic}\nolimits}

\newcommand{\ch}{\mathop{\rm ch}\nolimits}

\newcommand{\td}{\mathop{\rm td}\nolimits}
\newcommand{\Ext}{\mathop{\rm Ext}\nolimits}

\newcommand{\Coh}{\mathop{\rm Coh}\nolimits}

\newcommand{\cneq}{\mathrel{\raise.095ex\hbox{:}\mkern-4.2mu=}}
\newcommand{\eqcn}{\mathrel{=\mkern-4.5mu\raise.095ex\hbox{:}}}

\newcommand{\Aut}{\mathop{\rm Aut}\nolimits}

\newcommand{\DT}{\mathop{\rm DT}\nolimits}

\newcommand{\Imm}{\mathop{\rm Im}\nolimits}

\begin{document}
\maketitle
\begin{abstract}
We study a conjectural relationship 
among Donaldson-Thomas type invariants 
on Calabi-Yau 3-folds
counting 
torsion sheaves supported on 
ample divisors,
ideal sheaves of curves and 
Pandharipande-Thomas's stable pairs. 
The conjecture is a mathematical formulation 
of Denef-Moore's
formula derived in the study of 
Ooguri-Strominger-Vafa's conjecture 
relating black hole entropy and 
topological string.
The main result of this paper is to prove our conjecture 
assuming a conjectural Bogomolov-Gieseker type 
inequality proposed by Bayer, Macri and the author.   
\end{abstract}
\section{Introduction}
For a smooth projective Calabi-Yau 3-fold $X$, 
the Donaldson-Thomas (DT) invariant is 
introduced in~\cite{Thom}
as a holomorphic analogue of Casson invariants 
on real 3-manifolds. It counts 
 (semi)stable coherent sheaves on 
$X$, and its rank one theory 
is conjectured to be equivalent 
to the Gromov-Witten theory on $X$
by Maulik-Nekrasov-Okounkov-Pandharipande (MNOP)~\cite{MNOP}. 
Rather recently, wall-crossing phenomena 
of DT type invariants has drawn much 
attention, and its general theory 
is established by Joyce-Song~\cite{JS}, 
Kontsevich-Soibelman~\cite{K-S}. 
Applying the wall-crossing formula
to rank one DT type invariants, some
geometric applications related to 
MNOP conjecture have been obtained. 
 (cf.~\cite{BrH},~\cite{StTh},~\cite{Tcurve1},~\cite{Tolim2}.)

In this paper, we focus on
DT invariants 
counting 
torsion sheaves supported on ample 
divisors in 
$X$, and discuss their relationship 
to rank one DT type invariants
via wall-crossing phenomena. 
In string theory, counting 
sheaves on sufficiently ample 
divisors is interesting since it
is related to Strominger-Vafa's black hole entropy 
in terms of D-brane microstates~\cite{SV}. 
There are several physics
articles in which such 
sheaves or counting invariants 
are discussed,
(cf.~\cite{DM}, \cite{AM}, \cite{GGHS05}, \cite{GSY06}, \cite{DM07},)
while there has been no pure mathematical 
treatment of this subject. 

The 
work of this paper is 
motivated by Denef-Moore's 
approach~\cite{DM}
toward Ooguri-Strominger-Vafa
(OSV)
conjecture~\cite{OSV} relating black hole entropy
and topological string
on Calabi-Yau 3-folds. 
The idea of Denef-Moore~\cite{DM} is 
to investigate the decay of \rm{D}4 branes 
wrapping ample divisors in $X$ into 
D6-anti-D6 bound states on $X$. 
Through some physical arguments, 
they derive a formula relating 
indices of BPS 
\rm{D}4 branes on $X$ to those of 
D6-anti-D6 bound states on $X$.   
Roughly speaking, their formula is written as
\begin{align}\label{D4-D6}
\zZ_{\rm{D}4} \sim \zZ_{\rm{D}6-\overline{\rm{D}6}}
\end{align}
where the LHS and RHS are the 
D4-brane 
partition function, 
D6-anti-D6 bound state
partition function respectively. 
The $`\sim'$ in (\ref{D4-D6}) means that both sides are 
`approximated' in some sense. 
The purpose of this paper is 
summarized as follows: 
\begin{itemize}
\item
We formulate the relationship (\ref{D4-D6})
as a mathematically precise conjecture
in terms of 
DT type invariants.
\item We prove 
the above conjectural formula 
assuming
a conjectural  
Bogomolov-Gieseker type inequality
for \textit{tilt semistable objects}
in the derived category of coherent sheaves, 
proposed by Bayer, Macri and the author~\cite{BMT}. 
\end{itemize}
More precisely, the mathematical counterpart 
of the LHS of (\ref{D4-D6})
is
the generating series of DT invariants 
counting torsion sheaves supported on ample 
divisors in $X$, the RHS of (\ref{D4-D6})
is a certain generating series of 
the products of rank one DT invariants and 
Pandharipande-Thomas (PT) stable 
pair invariants~\cite{PT}. 
The wall-crossing phenomena of the 
tilt stability in~\cite{BMT}
is relevant to show the relationship (\ref{D4-D6}). 
In the proof, we will see how the 
conjectural Bogomolov-Gieseker type inequality
in~\cite{BMT} is effectively applied.

\subsection{Donaldson-Thomas invariants}
Let $X$ be a smooth projective Calabi-Yau 
3-fold over $\mathbb{C}$, i.e. 
\begin{align*}
\bigwedge^{3} T_{X}^{\vee} \cong \oO_X, \quad
H^1(X, \oO_X)=0. 
\end{align*}
Given an element, 
\begin{align*}
(r, D, \beta, n) 
\in H^0(X) \oplus H^2(X) \oplus H^4(X) \oplus H^6(X)
\end{align*}
and an ample divisor $H$ in $X$, 
we have the associated DT
invariant~\cite{Thom}, \cite{JS}, \cite{K-S}, 
\begin{align}\label{intro:DT}
\DT_H(r, D, \beta, n) \in \mathbb{Q}.
\end{align}
The invariant (\ref{intro:DT}) counts 
$H$-semistable coherent sheaves $E$ 
on $X$ satisfying 
\begin{align*}
\ch(E)
=(r, D, \beta, n). 
\end{align*} 
We are interested in
DT invariants in the following two 
cases: 

(i) $r=0$ and $D=mH$ for $m\in \mathbb{Z}_{>0}$. 
In this case, the invariant
\begin{align}\label{intro:(i)}
\DT_H(0, mH, -\beta, -n) \in \mathbb{Q}
\end{align}
counts 
$H$-semistable torsion 
sheaves supported on 
some ample divisor $P\subset X$.  
In string theory, such sheaves 
correspond to \rm{D}4-branes wrapping a divisor $P$.

(ii) $r=1$ and $D=0$. 
In this case, the invariant 
\begin{align}\label{intro:I}
I_{n, \beta} \cneq \DT_H(1, 0, -\beta, -n) \in \mathbb{Z}
\end{align}
counts
ideal 
sheaves of subschemes $C \subset X$, satisfying 
\begin{align*}
[C]=\beta, \quad \chi(\oO_C)=n. 
\end{align*}
Here $\beta$ and $n$ are interpreted as elements 
of $H_2(X, \mathbb{Z})$, $H_0(X, \mathbb{Z}) \cong \mathbb{Z}$
 respectively by Poincar\'e duality.
In string theory, such sheaves correspond to 
\rm{D}6-branes.  
The invariants (\ref{intro:I})
count curves in $X$, and their generating series
is expected to be coincide with the generating series of 
Gromov-Witten invariants of $X$
after some variable change by MNOP~\cite{MNOP}. 

The
Pandharipande-Thomas~\cite{PT} 
stable pair invariants
also count curves in $X$, which 
are closely related to the invariants (\ref{intro:I}). 
For $\beta \in H_2(X, \mathbb{Z})$ and $n\in \mathbb{Z}$, 
the PT invariant is denoted by 
\begin{align}\label{intro:PT}
P_{n, \beta} \in \mathbb{Z}. 
\end{align}
The objects which contribute to the invariants
(\ref{intro:PT})
are not necessary sheaves but
two term complexes of the form 
\begin{align}\label{intro:twoterm}
\cdots \to 0 \to \oO_X \stackrel{s}{\to} F \to 0 \to \cdots
\end{align} 
where $F$ is a pure one dimensional sheaf
satisfying 
\begin{align*}
[F]=\beta, \quad \chi(F)=n
\end{align*}
and $s$ is surjective in dimension one. 
In~\cite{PT}, the generating series of the invariants (\ref{intro:PT})
is conjectured to be related to that of (\ref{intro:I}) 
via wall-crossing phenomena 
in the derived category. 
This conjecture, called DT/PT correspondence, 
 is proved 
in~\cite{StTh},~\cite{Tcurve1}
at the Euler characteristic level and in~\cite{BrH} for 
the honest DT invariants. 

\subsection{Conjecture}
We formulate a conjecture relating invariants (\ref{intro:(i)}), 
(\ref{intro:I})
and (\ref{intro:PT}) in terms of generating 
series, following the idea of~\cite{DM}.
In what follows, we fix an ample divisor $H$ 
in $X$. 

The generating series of the invariants (\ref{intro:(i)}), 
i.e. D4-brane counting, is defined by  
\begin{align}\label{Z4}
\zZ_{\rm{D}4}^{m}(x, y) \cneq 
\sum_{\beta, n} \DT_H(0, mH, -\beta, -n) x^n y^{\beta}. 
\end{align}
As for the generating series of 
the invariants (\ref{intro:I}), (\ref{intro:PT}), 
we first define the 
following \textit{cut off} generating series: 
\begin{align*}
I^{m, \epsilon}(x, y) &\cneq 
\sum_{(\beta, n) \in C(m, \epsilon)} I_{n, \beta}x^n y^{\beta} \\
P^{m, \epsilon}(x, y) &\cneq 
\sum_{(\beta, n) \in C(m, \epsilon)} P_{n, \beta}x^n y^{\beta}.
\end{align*}
Here $m\in \mathbb{Z}_{>0}$, $\epsilon \in \mathbb{R}_{>0}$, and 
$C(m, \epsilon)$ is defined to be
\begin{align*}
C(m, \epsilon) \cneq \{ (\beta, n) \in H_2(X) \oplus H_0(X) :
\beta \cdot H < \epsilon m^2, \ \lvert n \rvert < \epsilon m^3\}.
\end{align*}
Following~\cite[Equation~(6.94)]{DM}, 
the generating series related to the RHS of (\ref{D4-D6}) 
is defined by 
\begin{align}\notag
\zZ_{\rm{D}6-\overline{\rm{D}6}}^{m, \epsilon}(x, y, z) \cneq 
&\sum_{D_2-D_1=mH}
x^{\frac{1}{6}(D_1^3-D_2^3)}
y^{\frac{1}{2}(D_1^2-D_2^2)}
z^{\frac{H^3}{6}m^3 + \frac{H \cdot c_2(X)}{12}m} \\
&\label{intro:ZD6} \quad I^{m, \epsilon}(xz^{-1}, x^{D_2}y z^{-mH})
P^{m, \epsilon}(x z^{-1}, x^{-D_1}y^{-1} z^{-mH}). 
\end{align}
Here in the sum (\ref{intro:ZD6}),
$D_1$, $D_2$ are divisors in $X$.  
The
relationship (\ref{D4-D6}) 
can be formulated as (\ref{intro:D4=D6}) in the following conjecture: 
\begin{conj}\label{intro:conj}

(i) The invariant $\DT_H(0, mH, -\beta, -n)$
vanishes  
unless 
\begin{align}\notag
-\frac{H^3}{24}m^3 \le 
n+\frac{(\beta \cdot H)^2}{2mH^3}. 
\end{align}

(ii) For any $\xi \ge 1$, there are 
$\mu>0$, $\delta>0$ and a constant 
$m(\xi, \mu)>0$ which 
depends only on $\xi$, $\mu$ such that
for any $m>m(\xi, \mu)$,  
we have the equality of the generating series,
\begin{align}\label{intro:D4=D6}
\zZ_{\rm{D}4}^{m}(x, y)
= \left. \frac{\partial}{\partial z} \zZ_{\rm{D}6-\overline{\rm{D}6}}^{m, \epsilon=\frac{\delta}{m^{\xi}}}(x, y, z) \right|_{z=-1}
\end{align}
modulo terms of $x^n y^{\beta}$
with 
\begin{align}\notag
-\frac{H^3}{24}m^3
 \left(1-\frac{\mu}{m^{\xi}} \right)
\le n +\frac{(\beta \cdot H)^2}{2mH^3}.
\end{align}
\end{conj}

\begin{rmk}
If $\Pic(X)$ is generated by $\oO_X(H)$, then 
$\mu$, $\delta$ may be taken as follows: 
if $\xi >1$, then $\mu$ is any
positive real number and if $\xi=1$, 
then $\mu$ is any real number satisfying 
$0<\mu< 3/2$. In both cases, $\delta$
is taken to be $\delta=\mu H^3/2$. 
See Theorem~\ref{thm:goal} for the detail. 
\end{rmk}

The formula (\ref{intro:D4=D6}) 
is a mathematical 
formulation of~\cite[Equation~(6.101)]{DM}, 
which plays an important role in the
study of the OSV conjecture~\cite{OSV}
in~\cite{DM}. 
The OSV conjecture states that, 
\begin{align}\label{OSV:physics}
\zZ_{\rm{BH}} \sim \lvert \zZ_{\rm{top}} \rvert^2
\end{align}
where the LHS is the
black hole partition function
and the RHS is the topological 
string partition function. Although 
a mathematically precise formulation of 
the OSV conjecture is not yet available, 
the formula (\ref{intro:D4=D6})
may give an intuition
of the relationship 
(\ref{OSV:physics}): 
the LHS of (\ref{intro:D4=D6})
counts massive D-branes for $m\gg 0$, 
hence it has to do with black holes.
On the other hand, 
the RHS of (\ref{intro:D4=D6})
involves products of DT type 
curve counting, hence
the square of the topological string, 
via MNOP conjecture~\cite{MNOP}.

\subsection{Main result}
Our main result is to 
prove Conjecture~\ref{intro:conj}, 
assuming two mathematical  
conjectures which are 
not yet proven. One of them is 
the conjectural Bogomolov-Gieseker type inequality
for tilt (semi)stable objects proposed in~\cite{BMT}. 
The other one is the conjectural property on the 
local moduli space of objects in the derived category. 

The former conjecture is stated as follows: let 
us take $B, \omega \in H^2(X, \mathbb{Q})$
so that $\omega$ is an ample $\mathbb{Q}$-divisor. 
Then we can construct the tilt of $\Coh(X)$, 
\begin{align*}
\bB_{B, \omega} \subset D^b \Coh(X)
\end{align*}
which is a certain abelian subcategory in the 
derived category of coherent sheaves. 
(cf.~Subsection~\ref{subsec:tilt}.)
In~\cite{BMT}, we constructed the slope function
$\nu_{B, \omega}$ on $\bB_{B, \omega}$ to be
\begin{align*}
\nu_{B, \omega}(E) \cneq \left\{ 
\begin{array}{cc}
\frac{\Imm Z_{B, \omega}(E)}
{\omega^2 \ch^{B}_1(E)}, & \omega^2 \ch^{B}_1(E) \neq 0 \\
\infty, & \omega^2 \ch^{B}_1(E)=0. 
\end{array}
\right. 
\end{align*}
Here $\ch^{B}(E)$ is the twisted Chern 
character $\ch^{B}(E) \cneq e^{-B}\ch(E)$, 
and $Z_{B, \omega}(E)$ is the central charge 
near the large volume limit in terms of string theory, 
\begin{align*}
Z_{B, \omega}(E) \cneq -\int_{X} e^{-i\omega}\ch^{B}(E). 
\end{align*}
The above slope function determines $\nu_{B, \omega}$-stability 
on $\bB_{B, \omega}$, which was called tilt stability in~\cite{BMT}. 
The following is the main conjecture in~\cite{BMT}. 
\begin{conj}{\bf \cite[Conjecture~1.3.1]{BMT}}
\label{conj:BG:intro}
For any tilt semistable object $E \in \bB_{B, \omega}$
with $\nu_{B, \omega}(E)=0$, the following inequality holds:
\begin{align}\label{intro:ineq}
\ch^{B}_3(E) \le \frac{1}{18}\omega^2 \ch^{B}_1(E).
\end{align}
\end{conj}
The Bogomolov-Gieseker type inequality for
tilt semistable objects
was first considered in order to
construct Bridgeland's stability conditions~\cite{Brs1} on 
projective 3-folds. For that purpose, it is enough 
to know the weaker inequality,
(cf.~\cite[Conjecture~3.2.7]{BMT},)
\begin{align}\label{weak}
\ch^{B}_3(E) < \frac{1}{2}\omega^2 \ch^{B}_1(E)
\end{align}
so the inequality (\ref{intro:ineq})
is stronger than the one required for the 
construction of Bridgeland stability. 
The stronger version (\ref{intro:ineq}) 
was conjectured from purely mathematical ideas
and observations as we explained in~\cite[Section~2]{BBMT}.  
One of the mathematical results which makes 
the inequality (\ref{intro:ineq}) reasonable is that, 
if we assume (\ref{intro:ineq}), then 
Fujita's conjecture on adjoint line bundles~\cite{Fujita} 
for 3-folds
follows~\cite{BBMT}. On the other hand, 
 this fact also shows that proving
the inequality (\ref{intro:ineq}) is very difficult, 
since a complete proof of 
Fujita's conjecture on 3-folds is 
still beyond the current 
knowledge of birational geometry. 

Another conjecture we assume is 
much more technical.
Let $\mM$ be the moduli stack of 
objects $E \in \bB_{B, \omega}$, 
which can be shown to be an algebraic stack over 
$\mathbb{C}$. 
We expect that $\mM$ is analytic locally written as a 
critical locus of some holomorphic function on 
a complex manifold, up to 
gauge action, 
as in the following conjecture:  
\begin{conj} {\bf \rm{(Conjecture~\ref{conj:CS})}}
\label{conj:CS:intro}
 For
any $[E] \in \mM$, 
let $G$ be a maximal reductive 
subgroup in $\Aut(E)$.
Then there exists a $G$-invariant analytic 
open neighborhood $V$ of $0$ in 
$\Ext^1(E, E)$, 
a $G$-invariant holomorphic function $f\colon V\to \mathbb{C}$
with $f(0)=df|_{0}=0$, and a smooth morphism 
of complex analytic stacks
\begin{align*}
\Phi \colon [\{df=0\}/G] \to \mM
\end{align*}
of relative dimension $\dim \Aut(E)- \dim G$. 
\end{conj}
The above conjecture 
is a derived category version of~\cite[Theorem~5.3]{JS}
and proved if $E \in \Coh(X)$ in~\cite[Theorem~5.3]{JS}
using an analytic method. 
Also a similar result is already announced by Behrend-Getzler~\cite{BG}. 
The result of Conjecture~\ref{conj:CS:intro}
will be needed in order to apply the wall-crossing formula
of DT type invariants~\cite{JS}, \cite{K-S} in the 
derived category setting. 
We strongly believe that Conjecture~\ref{conj:CS:intro} is true, 
and leave its full detail to a future publication. 

Our main result is the following: 
\begin{thm}{\bf \rm{(Corollary~\ref{cor:conji}, 
Theorem~\ref{thm:goal})}}
\label{intro:main}
Let $X$ be a smooth projective Calabi-Yau 3-fold 
such that $\Pic(X)$ is generated by $\oO_X(H)$ for an 
ample divisor $H$ in $X$. Assume that 
$X$ satisfies Conjecture~\ref{conj:BG:intro} 
and Conjecture~\ref{conj:CS:intro}. 
Then $(X, H)$ satisfies Conjecture~\ref{intro:conj}. 
\end{thm}
If we do not assume Conjecture~\ref{conj:CS:intro}
and just assume Conjecture~\ref{conj:BG:intro}, 
we still have the Euler characteristic version of 
Theorem~\ref{intro:main} as in the 
works~\cite{StTh},~\cite{Tcurve1},~\cite{Tolim2}.
(cf.~Theorem~\ref{thm:goal:euler}.) 
We also remark that assuming the weaker 
inequality (\ref{weak})
does not imply 
any reasonable result. The value $`18'$ in the denominator 
of the RHS of (\ref{intro:ineq}) is crucial in the argument. 

The result of Theorem~\ref{intro:main}
is a sort of results as in~\cite{BBMT}, i.e. 
assuming the inequality (\ref{intro:ineq}) 
yields a reasonable result predicted 
by some other works, which convinces us 
the validity of  
Conjecture~\ref{conj:BG:intro}.  
Indeed, 
the argument of the proof of Theorem~\ref{intro:main}
is closely related to that of~\cite[Theorem~1.1]{BBMT}: 
the tilt semistable objects
we discuss in this paper contain an object
$E \in \bB_{B, \omega}$ which fits into 
 a distinguished triangle, 
\begin{align*}
\oO_X(-mH)[1] \to E \to I_{Z}
\end{align*}
where $Z \subset X$ is a zero dimensional closed 
subscheme. 
(cf.~Subsection~\ref{subsec:Exam}.)
Such an object played an important role in~\cite{BBMT} in 
proving Fujita's conjecture~\cite{Fujita}
 assuming Conjecture~\ref{conj:BG:intro}. 
Furthermore, as we will explain
in Subsection~\ref{subsec:Exam}, 
a curve which appears in the RHS of (\ref{intro:D4=D6}) 
may be defined by a multiplier ideal 
sheaf of some log canonical $\mathbb{Q}$-divisor in $X$. 
This is an important notion 
in the study of 
Fujita's conjecture
and modern birational geometry~\cite{EL}, \cite{Kaw97}, \cite{Hel97}.  
These observations seem to give a surprising connection between 
two different research fields: 
Fujita's conjecture in birational geometry and 
black hole entropy in string theory. 

In order to obtain an intuition how the invariants
(\ref{intro:(i)}), 
(\ref{intro:I})
and (\ref{intro:PT}) are related 
as in
(\ref{intro:D4=D6}), the 
following geometric observation 
may be helpful: 
if $E$ is 
a torsion sheaf contributing to the invariant (\ref{intro:(i)}), 
and it is supported on a smooth member $P \in \lvert mH \rvert$, 
then it is written as 
\begin{align*}
E \cong i_{\ast}(\lL \otimes I_Z)
\end{align*}
where $i \colon P \hookrightarrow X$ is the inclusion, 
$\lL \in \Pic(P)$ and $Z \subset P$ is a 
zero dimensional closed subscheme. 
The line bundle $\lL$ is written as 
$\oO_P(C_2-C_1)$, where $C_1$, $C_2$
are curves in $P$ which do not have common 
irreducible components. 
If $Z$ is disjoint from $C_1$, $C_2$, there is a 
distinguished triangle, 
\begin{align}\label{decay}
I_{C_1 \cup Z} \to E \to \oO_X(-mH) \otimes 
\mathbb{D}(\oO_X \to \oO_{C_2}(C_1 \cap C_2))[1]. 
\end{align}
Here $\mathbb{D}(\ast)$ is the derived dual 
of the complex $\ast$. 
In the sequence (\ref{decay}), 
the left object contributes to (\ref{intro:I})
and the right object contributes to (\ref{intro:PT}). 
In this way, we can see that the objects 
contributing to the invariants (\ref{intro:(i)}), 
(\ref{intro:I}), (\ref{intro:PT})
are related. 
In terms of string theory, the sequence (\ref{decay})
realizes the decay of 
the D4 brane $E$ into 
D6 brane $I_{C_1 \cup Z}$
and anti D6 brane $\oO_X(-mH) \otimes 
\mathbb{D}(\oO_X \to \oO_{C_2}(C_1 \cap C_2))[1]$, 
i.e. D4 $\to$ D6+$\overline{\rm{D}6}$, 
which plays an important role in Denef-Moore's work~\cite{DM}. 

In general the sheaf $E$ may be supported on 
a singular divisor $P \subset X$, 
 which may be even non-reduced. So the above naive geometric 
argument is not applied in a general case. 
However we can use the wall-crossing argument
as we explain in the next subsection.

\subsection{Outline of the proof of Theorem~\ref{intro:main}}
We give an outline of 
the proof of Theorem~\ref{intro:main} (ii).
For a given element, 
\begin{align*}
v=(0, mH, -\beta, -n) \in H^{\ast}(X, \mathbb{Q}) 
\end{align*}
where $H$ is an ample generator of 
$\Pic(X)$ and $m\in \mathbb{Z}_{>0}$, 
we take $B\in H^2(X, \mathbb{Q})$ so that we have 
\begin{align*}
e^{-B}v= \left( 0, mH, 0, \frac{H^3}{24}m^3(1-\eta) \right).
\end{align*}
Here $\eta$ is given by 
\begin{align}\label{intro:eta}
\eta \cneq \left(n+ \frac{(\beta \cdot H)^2}{2mH^3} 
+\frac{H^3}{24}m^3 \right)
 / \frac{H^3}{24}m^3. 
\end{align}
For each $t\in \mathbb{R}_{>0}$,
we consider tilt stability on $\bB_{B, H}$
with respect to $\nu_{t} \cneq \nu_{B, tH}$. 
For each $t\in \mathbb{R}_{>0}$, we consider 
the moduli stack, 
\begin{align}\label{intro:DTt}
\mM_{t}^{ss}(v)
\end{align}
which parameterizes 
$\nu_{t}$-semistable objects
$E \in \bB_{B, H}$ satisfying 
$\ch(E)=v$. 
There is a wall and chamber structure 
on $\mathbb{R}_{>0}$, 
 the parameter space of $t$, 
such that the moduli stack (\ref{intro:DTt})
is constant on a chamber but jumps at 
walls. 
The behavior of the moduli stack (\ref{intro:DTt})
under the wall-crossing 
is described 
in terms of the stack theoretic Hall algebra
of $\bB_{B, H}$, 
as in~\cite{Joy4}, \cite{JS}, \cite{K-S}. 
We can show that, for $t\gg 0$, we have 
\begin{align}\label{M:RHS}
\mM_{t}^{ss}(v) =\mM_{H}^{ss}(v)
\end{align}
where the RHS is the moduli stack
which defines the invariant (\ref{intro:(i)}). 
On the other hand, if we assume Conjecture~\ref{conj:BG:intro}, 
we can show that  
\begin{align*}
\mM_{t}^{ss}(v)=\emptyset
\end{align*}
when $\eta<1$ and $0<t \ll 1$. 
By the above observations, 
the moduli stack $\mM_{H}^{ss}(v)$
can be 
described in terms of the wall-crossing factors
in the Hall algebra.  
The key point is that,
if we assume that 
Conjecture~\ref{conj:BG:intro}
is true 
and the rational number $\eta$ in 
(\ref{intro:eta}) is suitably small, 
then the objects which contribute to the
wall-crossing factors are
one of the following forms:  
\begin{align}\label{J:factor}
\oO_X(m_1 H) \otimes I_{C}, \quad 
 \oO_X(m_2 H) \otimes 
\mathbb{D}(\oO_X \to F)[1]. 
\end{align}
Here $m_i \in \mathbb{Z}$, 
$C \subset X$ is a subscheme with $\dim C \le 1$
and $\mathbb{D}(\oO_X \to F)$ is the derived dual 
of the two term complex $(\oO_X \to F)$
determined by a PT stable pair (\ref{intro:twoterm}).  
Hence it turns out that the stack $\mM_H^{ss}(v)$
is related to the moduli stacks of 
the objects of the form (\ref{J:factor}) in the Hall algebra.
By the wall-crossing formula of DT type invariants~\cite{JS}, 
\cite{K-S}, 
the above relationship of the moduli stacks 
can be integrated to a relationship 
among DT type invariants
(\ref{intro:(i)}), (\ref{intro:I})
and (\ref{intro:PT}),
if we assume Conjecture~\ref{conj:CS:intro}. 
The formula (\ref{intro:D4=D6})
is the resulting relationship among these
invariants
in terms of generating series. 

The conjectural inequality (\ref{intro:ineq}) is
used to 
investigate the 
behavior of numerical classes under the wall-crossing. 
For instance, if $\eta$ satisfies
$0\le \eta \ll 1$, then the resulting 
wall-crossing
factors (\ref{J:factor})
satisfy 
\begin{align}\label{intro:extreme}
([C], \chi(\oO_C)) \in C(m, \epsilon), \ 
([F], \chi(F)) \in C(m, \epsilon)
\end{align}
for $0<\epsilon \ll 1$. 
The above property 
corresponds to the 
\textit{extreme polar state conjecture}, 
which was a conjecture in~\cite{DM} even in
the physics sense. 
The inequality (\ref{intro:ineq})
is also used to show that, if
$\eta$ further satisfies 
\begin{align*}
0\le \eta <\frac{\mu}{m^{\xi}}, \quad 
\mu \in \mathbb{R}_{>0}, \ 
 \ m\gg 0
\end{align*}
for some $\xi \ge 1$ and $\mu>0$, then 
the wall-crossing factors are 
contributed by the objects of the form (\ref{J:factor}). 
This property corresponds to the
 \textit{core dump exponent conjecture}, 
that is the real number $\xi_{cd}$
in~\cite{DM} satisfies 
$\xi_{cd} =1$. 
The above two physical conjectures (extreme polar state conjecture, 
core dump exponent conjecture) were relevant to approximate 
both sides of (\ref{OSV:physics}).  
The main contribution of this paper is 
to deduce these physical conjectures 
from the single conjectural inequality (\ref{intro:ineq}). 
\subsection{Plan of the paper}
The plan of the paper is as follows. 
In Section~\ref{sec:back}, we give some background 
of stability conditions and DT type invariants. 
In Section~\ref{sec:proof}, we give a proof of Theorem~\ref{intro:main}. 
In Section~\ref{sec:append}, we discuss 
the relevant wall-crossing phenomena, and give
an evidence to Conjecture~\ref{intro:conj} (i). 
\subsection{Acknowledgments}
The author is grateful to Arend Bayer and Kentaro Hori 
for valuable discussions and comments. 
He also thanks the referee for several helpful comments. 
This work is supported by World Premier 
International Research Center Initiative
(WPI initiative), MEXT, Japan. This work is also supported by Grant-in Aid
for Scientific Research grant (22684002), 
and partly (S-19104002),
from the Ministry of Education, Culture,
Sports, Science and Technology, Japan.
\section{Background}\label{sec:back}
In this section, we collect
some notions, results, which we 
will use in the proof of Theorem~\ref{intro:main}. 
In what follows, 
$X$ is a smooth projective Calabi-Yau 3-fold 
over $\mathbb{C}$, i.e. 
\begin{align*}
\bigwedge^{3}T_X^{\vee} \cong \oO_X, \quad 
H^1(X, \oO_X)=0. 
\end{align*}
\subsection{Twisted Gieseker stability}
We recall the classical notion of twisted
stability
on $\Coh(X)$ in the sense of Gieseker, 
which will be used in 
constructing DT invariants. 
For the detail of the non-twisted case, see~\cite{Hu}. 
We take an element, 
\begin{align}\label{Bomega}
B+i\omega \in H^2(X, \mathbb{C})
\end{align}
where $B$, $\omega$ are defined over $\mathbb{Q}$, 
and $\omega$ is ample.
For $E\in \Coh(X)$, 
the twisted Hilbert polynomial is
defined by 
\begin{align*}
\chi(E, n) &\cneq \int_X \ch^{B}(E) e^{n\omega} \td_X \\
&= a_d n^d + a_{d-1} n^{d-1} + \cdots, 
\end{align*}
where $a_i \in \mathbb{Q}$, $a_d \neq 0$
and $d=\dim \Supp(E)$. 
Here $\ch^{B}(E)$ is the twisted Chern character, 
\begin{align}\label{tCh}
\ch^{B}(E) \cneq e^{-B} \ch(E) \in H^{\ast}(X, \mathbb{Q}).
\end{align}
The reduced twisted Hilbert polynomial is 
defined by 
\begin{align*}
\overline{\chi}(E, n) \cneq \chi(E, n)/a_d. 
\end{align*}
\begin{defi}
An object $E\in \Coh(X)$ is $B$-twisted 
$\omega$-(semi)stable if the following conditions hold:
\begin{itemize}
\item $E$ is a pure sheaf, i.e. there is no
subsheaf $0\neq F \subsetneq E$ with 
$\dim \Supp(F)<\dim \Supp(E)$. 
\item For any subsheaf $0\neq F \subsetneq E$, 
we have for $n\gg 0$, 
\begin{align*}
\overline{\chi}(F, n)<(\le) \overline{\chi}(E, n). 
\end{align*}
\end{itemize}
\end{defi}
If $B=0$, then 
the $B$-twisted $\omega$-(semi)stable sheaves
are called $\omega$-(semi)stable sheaves. 

\subsection{Twisted slope stability}
Here we recall the
notion of 
twisted slope
stability 
determined by an element
(\ref{Bomega}), 
which is coarser than the twisted 
Gieseker stability. 
For $d\in \mathbb{Z}_{\ge 0}$, let 
$\Coh_{\le d}(X) \subset \Coh(X)$ be the 
subcategory defined by 
\begin{align*}
\Coh_{\le d}(X) \cneq \{ E\in \Coh(X) : 
\dim \Supp(E) \le d\}.
\end{align*}
The $B$-twisted $\omega$-slope function
\begin{align*}
\mu_{B, \omega, d} \colon 
\Coh_{\le d}(X) \to \mathbb{Q} \cup \{ \infty\}
\end{align*}
is defined as follows: 
if $E \in \Coh_{\le d-1}(X)$, we set 
$\mu_{B, \omega, d}(E)=\infty$. Otherwise
we set 
\begin{align*}
\mu_{B, \omega, d}(E) \cneq 
\frac{\omega^{d-1} \cdot \ch^{B}_{4-d}(E)}{\omega^d \cdot \ch^{B}_{3-d}(E)}.
\end{align*}
The notion of $\mu_{B, \omega, d}$-stability is defined as follows:  
\begin{defi}\label{G-stab}
An object $E \in \Coh_{\le d}(X)$ is 
$\mu_{B, \omega, d}$-(semi)stable if for
any exact sequence 
$0 \to F \to E \to G \to 0$ 
in $\Coh_{\le d}(X)$, we have 
\begin{align*}
\mu_{B, \omega, d}(F)<(\le) \mu_{B, \omega, d}(G). 
\end{align*}
\end{defi}
The above stability condition is called
a $B$-twisted $\omega$-slope stability condition. 
Note that if $B$ is proportional to $\omega$, then 
$\mu_{B, \omega, d}$-stability coincides with 
$\mu_{\omega, d} \cneq \mu_{0, \omega, d}$-stability. 
Also if $d=3$, we just set
$\mu_{B, \omega} \cneq \mu_{B, \omega, 3}$. 
\begin{rmk}\label{rmk:check}
It is easy to check that
the twisted slope stability is 
coarser than the twisted stability. 
Namely
if $E$ is a $d$-dimensional sheaf, 
then $E\in \Coh_{\le d}(X)$ is, omitting
notation $B$-twisted and $\omega$-, 
\begin{align*}
\mbox{\rm{slope stable}}
\Rightarrow \mbox{\rm{stable}}
\Rightarrow \mbox{\rm{semistable}}
\Rightarrow \mbox{\rm{slope semistable}}. 
\end{align*}
\end{rmk}

\subsection{Tilt stability}\label{subsec:tilt}
The notion of tilt stability in~\cite{BMT}
is defined in the abelian 
category obtained as a tilt of $\Coh(X)$. 
For an element (\ref{Bomega}),  
we set the following subcategories of $\Coh(X)$: 
\begin{align*}
\tT_{B, \omega} &\cneq 
\left\langle 
E : \begin{array}{c} E 
\mbox{ is a torsion sheaf or torsion free }  \\
\mu_{B, \omega}
\mbox{-semistable
sheaf with }
\mu_{B, \omega}(E)>0.
\end{array}
\right\rangle \\
\fF_{B, \omega} &\cneq 
\left\langle
E: 
\begin{array}{c}
 E \mbox{ is a torsion free }
\mu_{B, \omega} \mbox{-semistable } \\
\mbox{ sheaf 
 with } \mu_{B, \omega}(E) \le 0
\end{array}
  \right\rangle.
\end{align*}
Here $\langle S \rangle$ is the 
smallest extension closed subcategory 
which contains $S$. 
The pair of subcategoris $(\tT_{B, \omega}, \fF_{B, \omega})$
is a torsion pair of $\Coh(X)$.
(cf.~\cite{HRS}.)
Its tilting is defined by 
\begin{align*}
\bB_{B, \omega} \cneq \langle 
\fF_{B, \omega}[1], \tT_{B, \omega} \rangle
\subset D^b \Coh(X).  
\end{align*}
By a general theory of tilting~\cite{HRS}, 
 $\bB_{B, \omega}$ is the heart of a bounded t-structure 
on $D^b \Coh(X)$, hence an abelian category. 
\begin{rmk}
For any object $E\in \bB_{B, \omega}$, 
we have $\ch^{B}_1(E)\omega^2 \ge 0$. 
In particular if $F$ is a subobject of $E$ in 
$\bB_{B, \omega}$, we have 
$\ch^{B}_1(E)\omega^2 \ge \ch^{B}_1(F)\omega^2$.
\end{rmk}
Let $Z_{B, \omega} \colon K(X) \to \mathbb{C}$
be the group homomorphism defined by 
\begin{align*}
Z_{B, \omega}(E) &\cneq -\int_{X}
\ch^{B}(E) e^{-i\omega} \\
&= \left(-\ch^{B}_3(E)+\frac{1}{2}\omega^2 
\ch^{B}_1(E)  \right) + i \left( 
\omega \ch^{B}_2(E) -\frac{1}{6} \omega^3 
\ch^{B}_0(E) \right). 
\end{align*}
The slope function 
\begin{align*}
\nu_{B, \omega} \colon \bB_{B, \omega} \to 
\mathbb{Q} \cup \{ \infty \}
\end{align*}
is defined as follows: 
if $\omega^2 \ch^{B}_1(E)=0$, then we
set $\nu_{B, \omega}(E)=\infty$. 
Otherwise we set
\begin{align*}
\nu_{B, \omega}(E) \cneq 
\frac{\Imm Z_{B, \omega}(E)}{\omega^2 \ch^{B}_1(E)}.
\end{align*}
Similarly to $\mu_{B, \omega, d}$-stability 
on $\Coh_{\le d}(X)$, the 
above slope function defines 
$\nu_{B, \omega}$-stability on $\bB_{B, \omega}$:
\begin{defi}
An object $E\in \bB_{B, \omega}$ is 
$\nu_{B, \omega}$-(semi)stable if for any exact sequence 
$0 \to F \to E \to G \to 0$ in 
$\bB_{B, \omega}$, we have 
\begin{align*}
\nu_{B, \omega}(F) <(\le) \nu_{B, \omega}(G). 
\end{align*}
\end{defi}
The above stability condition on $\bB_{B, \omega}$
is 
called tilt stability in~\cite{BMT}. 

\subsection{Bogomolov-Gieseker type inequalities}
We will use two kinds
of inequalities of 
Chern characters of tilt semistable 
objects in $\bB_{B, \omega}$. 
One of them is a generalization of 
classical Bogomolov-Gieseker inequality 
to tilt semistable objects, proved in~\cite[Corollary~7.3.2]{BMT}. 
The other one is a conjectural inequality 
proposed in~\cite[Conjecture~1.3.1]{BMT}. 

The first inequality is as follows: 
\begin{thm}{\bf \cite[Corollary~7.3.2]{BMT}}\label{thm:class}
For any $\nu_{B, \omega}$-semistable object $E \in \bB_{B, \omega}$,
we have the inequality,
\begin{align}\label{BG:thm1}
(\omega^2 \ch^{B}_1(E))^2 -2 (\omega^3 \ch^{B}_0(E))
(\omega \ch^{B}_2(E)) \ge 0. 
\end{align}
\end{thm}
Note that if $\Pic(X)$ is generated by 
$\oO_X(H)$ for an ample divisor $H$, the 
inequality (\ref{BG:thm1}) is equivalent to
\begin{align}\label{BG:thm2}
H (\ch^{B}_1(E))^2 -2 H \ch^{B}_0(E)\ch^{B}_2(E) \ge 0. 
\end{align}
The second inequality is given in the
following conjecture: 
\begin{conj}{\bf \cite[Conjecture~1.3.1]{BMT}}\label{conj:BMT}
For any $\nu_{B, \omega}$-semistable object
$E\in \bB_{B, \omega}$ with 
$\nu_{B, \omega}(E)=0$, we have the inequality,
\begin{align}\label{conj:ineq:ch3}
\ch^{B}_3(E) \le \frac{1}{18}\omega^2 \ch^{B}_1(E). 
\end{align}
\end{conj}
\begin{rmk}
The above conjecture is not restricted to 
Calabi-Yau 3-fold case. 
For instance, the inequality (\ref{conj:ineq:ch3})
is proved when $X=\mathbb{P}^3$
and $\omega^3 <3\sqrt{3}$. 
(cf.~\cite[Theorem~8.2.1]{BMT}.)
\end{rmk}
\begin{rmk}
There is the heart of 
a bounded t-structure $\aA_{B, \omega} \subset D^b \Coh(X)$
obtained as a tilting of $\bB_{B, \omega}$.
(cf.~\cite[Definition~3.2.5]{BMT}.)
If we assume Conjecture~\ref{conj:BMT}, 
then the pair 
\begin{align}\label{stab:A}
(Z_{B, \omega}, \aA_{B, \omega})
\end{align}
gives a Bridgeland's stability condition~\cite{Brs1} on 
$D^b \Coh(X)$. (cf.~\cite[Conjecture~3.2.6]{BMT}.)
We note that any tilt semistable object
$E$ with $\nu_{B, \omega}(E)=0$ is always 
semistable 
w.r.t. Bridgeland stability (\ref{stab:A})
with phase one. 
\end{rmk}

\subsection{Donaldson-Thomas invariants}\label{subsec:back:DT}
Let $H$ be an ample divisor 
in a Calabi-Yau 3-fold $X$. 
The DT invariant is an invariant counting
$H$-semistable sheaves
on $X$. 
For
an element 
\begin{align*}
(r, D, \beta, n) \in H^0(X) \oplus H^2(X) \oplus H^4(X)
\oplus H^6(X)
\end{align*}
it is denoted by 
\begin{align}\label{def:DT}
\DT_{H}(r, D, \beta, n) \in \mathbb{Q}. 
\end{align} 
The easier case to define
(\ref{def:DT})
is when any $H$-semistable sheaf
contributing to (\ref{def:DT}) is $H$-stable.  
Let
\begin{align}\label{M:stack}
\mM_{H}^{s}(r, D, \beta, n) \quad
(\mbox{ resp. }
\mM_H^{ss}(r, D, \beta, n) )
\end{align}
 be the moduli stack
of $H$-stable 
(resp.~$H$-semistable)
sheaves $E$ in the sense of Definition~\ref{G-stab},
satisfying 
$\ch(E)=(r, D, \beta, n)$.
The stacks (\ref{M:stack})
are known to be Artin stacks of finite type over $\mathbb{C}$. 
We have the 
$B\mathbb{C}^{\ast}$-bundle structure
\begin{align*}
\mM_{H}^{s}(r, D, \beta, n) \to M_H^{s}(r, D, \beta, n) 
\end{align*}
where $M_H^{s}(r, D, \beta, n)$ is a quasi-projective
 scheme over $\mathbb{C}$. 
When any sheaf 
$[E] \in \mM_H^{ss}(r, D, \beta, n)$
is $H$-stable, 
the scheme $M_H^{s}(r, D, \beta, n)$ is 
a projective $\mathbb{C}$-scheme. 
In this case, the DT invariant (\ref{def:DT}) 
is defined by 
\begin{align}\notag
\DT_{H}(r, D, \beta, n) &\cneq 
\int_{M_H^{s}(r, D, \beta, n)} \nu \ d \chi \\
\label{nu:DT}
&=\sum_{m\in \mathbb{Z}} m \chi(\nu^{-1}(m)). 
\end{align}
Here $\nu$
is the Behrend's constructible function~\cite{Beh}, 
\begin{align*}
\nu \colon M_H^{s}(r, D, \beta, n) \to \mathbb{Z}.
\end{align*}
Note that if $M_H^{s}(r, D, \beta, n)$ is non-singular, 
then $\DT_{H}(r, D, \beta, n)$ is 
the topological 
 Euler characteristic of $M_H^{s}(r, D, \beta, n)$
up to sign.

When there is a strictly $H$-semistable 
sheaf $[E] \in \mM_H^{ss}(r, D, \beta, n)$, the definition 
of (\ref{def:DT})
is much more complicated. In this case, 
(\ref{def:DT}) is defined by integrating the 
Behrend function over the 
`logarithm' of the moduli stack $\mM_H^{ss}(r, D, \beta, n)$
in the stack theoretic Hall algebra~\cite{JS}, \cite{K-S}. 
Since only DT invariants defined
as in (\ref{nu:DT}) appear in this paper,  
we omit the detail of the latter construction.  
 
\subsection{DT invariants counting torsion sheaves}
We are interested in DT invariants counting
sheaves supported on ample divisors in $X$, i.e. 
the invariants, 
\footnote{The sign of 
$\beta$ and $n$ is chosen in order to match the notation of 
rank one DT invariants (\ref{inv:I}) below. }
\begin{align}\label{DT:0m}
\DT_H(0, mH, -\beta, -n) \in \mathbb{Q}
\end{align}
for $m \in \mathbb{Z}_{>0}$. 
The generating series of the invariants (\ref{DT:0m})
is denoted by
\begin{align}\label{ZD4}
\zZ_{\rm{D}4}^{m}(x, y) \cneq 
\sum_{\beta, n} \DT_H(0, mH, -\beta, -n)x^{n} y^{\beta}. 
\end{align}
Although it is not an obvious problem to compute the series 
(\ref{ZD4}), its local 
version is very easy to compute as follows: 
\begin{exam}\label{exam:D4}
Let $P\in \lvert mH \rvert$ be a smooth member.
In the notation of the previous subsection, we have 
the subscheme, 
\begin{align*}
M_H^{s}(0, P, -\beta, -n) \subset M_H^{s}(0, mH, -\beta, -n)
\end{align*}
corresponding to stable sheaves supported on $P$. 
Then we have the local DT invariant, 
\begin{align}\label{loc:DT}
\DT_{H}(0, P, -\beta, -n)
\cneq \int_{M_H^{s}(0, P, -\beta, -n)} \nu \ d\chi
\end{align}
where $\nu$ is the Behrend function on 
$M_H^{s}(0, mH, -\beta, -n)$
restricted to $M_H^{s}(0, P, -\beta, -n)$. 
Note that $\nu$ is not the Behrend 
function on the subscheme
$M_H^{s}(0, P, -\beta, -n)$.
Let us compute the invariant (\ref{loc:DT}). 

We note that, 
although $M_H^{s}(0, mH, -\beta, -n)$ may not be 
projective, the scheme $M_H^{s}(0, P, -\beta, -n)$
is always projective. 
Indeed sheaves corresponding to points 
in $M_H^{s}(0, P, -\beta, -n)$
have the following form, 
\begin{align}\label{fo:form}
E\cong i_{\ast}(\lL \otimes I_{Z})
\end{align}
where $\lL \in \Pic(P)$, $Z \subset P$ is a 
zero dimensional closed subscheme, and 
$I_Z$ is the defining ideal of $Z$. 
The condition $\ch(E)=(0, mH, -\beta, -n)$
is equivalent to 
\begin{align*}
&\frac{H^2}{2}m^2 - i_{\ast}c_1(\lL)= \beta \\
&\lvert Z \rvert -\frac{H^3}{6}m^3 +\frac{Hc_1(\lL)}{2}m
- \frac{1}{2} c_1(\lL)^2 =n. 
\end{align*}
Here we have written $H|_{P}$ just as $H$ for simplicity,
and $i\colon P \hookrightarrow X$ is the inclusion.  
Since $H^1(\oO_P)=0$, we have the isomorphism 
$c_1 \colon \Pic(P) \stackrel{\cong}{\to} \mathrm{NS}(P)$, 
hence we have 
\begin{align}\label{MPb}
M_H^{s}(0, P, -\beta, -n)=
\coprod_{\begin{subarray}{c}
l\in \mathrm{NS}(P), \\
i_{\ast}l=\frac{1}{2}m^2 H^2 -\beta
\end{subarray}}
\Hilb^{\frac{H^3}{6}m^3 -\frac{Hl}{2}m + \frac{1}{2}l^2
+n}(P). 
\end{align}
Therefore $M_H^{s}(0, P, -\beta, -n)$
is projective, and the definition (\ref{loc:DT})
is an analogy of (\ref{nu:DT}) in the local case.

It is easy to see that the moduli space 
$M_H^{s}(0, mH, -\beta, -n)$ is smooth along 
$M_H^{s}(0, P, -\beta, -n)$ with dimension 
\begin{align*}
\dim \mathbb{P}(H^0(X, \oO_X(mH)))+ \dim \Hilb^{i}(P)
\end{align*}
for some $i \in \mathbb{Z}_{\ge 0}$. Therefore 
we have 
\begin{align}\label{nu:restrict}
\nu|_{M_H^{s}(0, P, -\beta, -n)}
=(-1)^{\frac{H^3}{6}m^3 + \frac{H c_2(X)}{12}m -1}.
\end{align}
By (\ref{MPb}) and (\ref{nu:restrict}), the generating series of 
local DT invariants (\ref{loc:DT}) is calculated as 
\begin{align*}
&\sum_{\beta, n} \DT(0, P, -\beta, -n)x^n y^{\beta} \\
&=(-1)^{\frac{H^3}{6}m^3 + \frac{H c_2(X)}{12}m -1}
\sum_{\begin{subarray}{c}
l\in \mathrm{NS}(P) \\
N\ge 0
\end{subarray}}
\chi(\Hilb^{N}(P))x^{N-\frac{H^3}{6}m^3 + \frac{Hl}{2}m
-\frac{1}{2}l^2}
y^{\frac{H^2}{2}m^2 -i_{\ast}l} \\
&=(-1)^{\frac{H^3}{6}m^3 + \frac{H c_2(X)}{12}m -1}
\prod_{N\ge 1}(1-x^N)^{-H^3 m^3 -Hc_2(X)m} \\
& \qquad \qquad \qquad \qquad \qquad \qquad \qquad \quad
\cdot\sum_{l \in \mathrm{NS}(P)}
x^{-\frac{H^3}{6}m^3+\frac{Hl}{2}m-\frac{1}{2}l^2}
y^{\frac{H^2}{2}m^2 -i_{\ast}l}.
\end{align*}
Here we have used the G$\ddot{\textit{o}}$ttsche's formula~\cite{Got}, 
\begin{align}\label{Gottsche}
\sum_{N\ge 0} \chi(\Hilb^{N}(P))x^N=\prod_{N\ge 1}
(1-x^N)^{-\chi(P)}. 
\end{align}
\begin{rmk}
Although it is easy to compute the local DT invariants 
when $P\in \lvert mH \rvert$ is non-singular
as in Example~\ref{exam:D4}, it is not obvious 
to compute the local invariants
when $P$ has singularities. For instance, 
the G$\ddot{\rm{o}}$ttsche
type formula (\ref{Gottsche}) is not known for singular surfaces. 
\end{rmk}
\end{exam}

\subsection{DT and PT invariants counting curves}\label{subsec:DTPT}
As we discussed in the introduction, our purpose is 
to relate the series (\ref{ZD4}) with the generating series
of two kinds
of DT type curve counting invariants in $X$. 
One of them is rank one DT invariant, and the 
other one is PT stable pair invariant~\cite{PT}. 
Both of these invariants depend on 
$\beta \in H_2(X, \mathbb{Z})$ and $n\in \mathbb{Z}$. 
Note that $\beta$ and $n$ are also interpreted 
as elements of $H^4(X)$ and $H^6(X)$ respectively 
by the Poincar\'e duality. 

The former invariant is defined by, 
\begin{align}\label{inv:I}
I_{n, \beta} \cneq \DT_{H}(1, 0, -\beta, -n). 
\end{align}
Note that any sheaf contributing to $I_{n, \beta}$
is of the form $I_{C}$ where $C \subset X$
is a closed subscheme with $\dim C\le 1$, 
$[C]=\beta$ and $\chi(\oO_C)=n$. 
(cf.~\cite{MNOP}.)
In particular the invariant (\ref{inv:I}) does not 
depend on a choice of $H$. 

The latter invariant roughly counts pairs of 
a curve and a divisor on it. This notion is formulated 
as stable pairs: by definition, a stable pair
is a pair 
\begin{align}\label{stable:pair}
(\oO_X \stackrel{s}{\to} F),
\end{align}
where $F$ is a pure one dimensional coherent sheaf on $X$,
and $s$ is surjective in dimension one. 
A typical example is given by 
$(\oO_X \stackrel{s}{\to} \oO_C(D))$, 
where $C \subset X$ is a smooth curve, 
$D \subset C$ is an effective divisor
and $s$ is a natural morphism. 
The moduli space of stable pairs (\ref{stable:pair})
satisfying $[F]=\beta$ and $\chi(F)=n$ is denoted by 
\begin{align*}
P_n(X, \beta). 
\end{align*}
The above moduli space is shown to be 
a projective scheme over $\mathbb{C}$. 
The PT invariant is defined by 
\begin{align}\label{inv:P}
P_{n, \beta} \cneq \int_{P_n(X, \beta)} \nu d\chi, 
\end{align}
where $\nu$ is the Behrend function on 
$P_n(X, \beta)$. Note that both of the invariants (\ref{inv:I}), (\ref{inv:P})
are integer valued. 

We will use
the following cut off generating series: 
\begin{align}\label{cut:I}
I^{m, \epsilon}(x, y) &\cneq 
\sum_{(\beta, n) \in C(m, \epsilon)} I_{n, \beta}x^n y^{\beta}, \\
\label{cut:II}
P^{m, \epsilon}(x, y) &\cneq 
\sum_{(\beta, n) \in C(m, \epsilon)} P_{n, \beta}x^n y^{\beta}.
\end{align}
Here $m\in \mathbb{Z}_{>0}$, $\epsilon \in \mathbb{R}_{>0}$, and 
$C(m, \epsilon)$ is defined to be
\begin{align}\label{Cme}
C(m, \epsilon) \cneq \{ (\beta, n) \in H_2(X) \oplus H_0(X) :
\beta \cdot H < \epsilon m^2, \ \lvert n \rvert < \epsilon m^3\}.
\end{align}
The relationship between 
$I_{n, \beta}$ and $P_{n, \beta}$ is 
conjectured in~\cite{PT}, 
proved 
in~\cite{StTh},~\cite{Tcurve1}
at the Euler characteristic level and in~\cite{BrH} for 
the honest DT invariants. 
This is formulated in terms of non-cut off generating 
series. If we define 
$I(x, y)$, $P(x, y)$ by formally putting 
$\epsilon=1$, $m=\infty$ 
to the series (\ref{cut:I}), (\ref{cut:II}) respectively, then 
we have, by~\cite{Tcurve1}, \cite{StTh}, \cite{BrH}, 
\begin{align}\label{DT/PT}
I(x, y)/M(-x)^{\chi(X)}=P(x, y). 
\end{align}
Here $M(x)$ is the MacMahon function 
\begin{align*}
M(x)=\prod_{n\ge 0}(1-x^n)^{-n}. 
\end{align*}

\section{Proof of Theorem~\ref{intro:main}}\label{sec:proof}
In this section, we give a proof of Theorem~\ref{intro:main}. 
In what follows, $X$ is a smooth projective 
Calabi-Yau 3-fold such that $\Pic(X)$ is 
generated by $\oO_X(H)$ for an ample divisor $H$ 
in $X$. 
\subsection{One parameter family of tilt stability}
We construct a one parameter family of tilt 
stability conditions 
depending on a choice of 
a numerical class, 
\begin{align}\label{v:class}
v=(0, mH, -\beta, -n) \in H^{\ast}(X, \mathbb{Q}). 
\end{align}
for $m\in \mathbb{Z}_{>0}$. 
In this section, we always fix 
$v$ as in (\ref{v:class}). 
We also fix $B \in H^2(X, \mathbb{Q})$ to be
\begin{align}\label{B:choice}
B=-\frac{\beta \cdot H}{mH^3}H. 
\end{align}
By a simple calculation, we have  
\begin{align}\label{calcu}
e^{-B}v=\left(0, mH, 0, \frac{H^3}{24}m^3(1-\eta) \right)
\end{align}
where $\eta$ is given by 
\begin{align}\label{def:eta}
\eta= \left(n+ \frac{(H\cdot \beta)^2}{2mH^3} + 
\frac{H^3}{24}m^3 \right)/\frac{H^3}{24}m^3.
\end{align}
In what follows, 
the twisted Chern character
$\ch^{B}(\ast)=e^{-B}\ch(\ast)$ 
is taken 
with respect to the $B$ chosen in 
(\ref{B:choice}).  
For $t\in \mathbb{R}_{>0}$, we set
\begin{align*}
\nu_{t} \cneq \nu_{B, tH} \colon 
\bB_{B, H} \setminus \{0\} \to \mathbb{R} \cup \{ \infty\}. 
\end{align*}
Note that $\bB_{B, tH}=\bB_{B, H}$ for any $t\in \mathbb{R}_{>0}$, 
so the above slope function is well-defined. 
Also any object $E\in \bB_{B, H}$ with 
$\ch(E)=v$, or equivalently 
$\ch^{B}(E)$ is equal to the RHS of (\ref{calcu}), 
satisfies $\nu_t(E)=0$. 

\subsection{Wall and chamber structure}\label{subsec:wall}
As proved in~\cite[Corollary~3.3.3]{BMT},
there is a wall and chamber structure 
on the set of tilt stability. 
In the setting of the previous subsection, 
there is a discrete subset of walls,  
\begin{align*}
\sS \subset \mathbb{R}_{>0},
\end{align*}
such that the set of 
$\nu_t$-semistable objects $E$
with $\ch(E)=v$
is constant 
when $t$ lies in a connected component of 
$\mathbb{R}_{>0} \setminus \sS$, 
but jumps at walls. 
We first show that there is no wall
when $t$ is bigger than $\sqrt{3}m/2$. 
\begin{lem}\label{wall:S}
We have $\sS \subset \{ t\in \mathbb{R}_{>0} : t\le \sqrt{3}m/2\}$. 
\end{lem}
\begin{proof}
Let us take $t_0 \in \sS$. Then there is a $\nu_{t_0}$-semistable 
object $E\in \bB_{B, H}$ with $\ch(E)=v$
 and an exact sequence in 
$\bB_{B, H}$
\begin{align*}
0 \to E_1 \to E \to E_2 \to 0 
\end{align*}
such that $\nu_{t_0}(E_1)=\nu_{t_0}(E_2)=0$ and 
\begin{align}\label{nu:ineq}
\nu_{t_0 +\epsilon}(E_1)< \nu_{t_0 +\epsilon}(E_2), \ 
\nu_{t_0 -\epsilon}(E_1)> \nu_{t_0 -\epsilon}(E_2)
\end{align}
for $0< \epsilon \ll 1$.
We write $\ch^{B}(E_i)=(r_i, D_i, \beta_i, n_i)$ 
for $i=1, 2$. Then the condition 
$\nu_{t_0}(E_i)=0$ is equivalent to 
\begin{align*}
-\frac{H^3}{6}t_0^2 r_i + \beta_i H=0. 
\end{align*}
Combined with (\ref{nu:ineq}), we obtain 
$r_1>0$ and $r_2<0$, hence $r_i^2 \ge 1$. 
 On the other hand, 
since $E_1$ and $E_2$ are $\nu_{t_0}$-semistable, 
we have the 
inequality 
$D_i^2 H \ge 2r_i \beta_i H$
by Theorem~\ref{thm:class}. 
Also since $D_1 +D_2=mH$, 
either $i=1$ or $2$ satisfies 
$D_i^2 H \le m^2 H^3/4$. 
For such $i$, we obtain 
\begin{align*}
\frac{H^3}{4}m^2 \ge 
D_i^2 H \ge 2r_i \beta_i H = \frac{H^3}{3}r_i^2 t_0^2 \ge \frac{H^3}{3}t_0^2
\end{align*}
for $i=1$ or $2$. 
The above inequalities imply
 $t_0 \le \sqrt{3}m/2$ as claimed. 
\end{proof}
By the above lemma, the region 
$\{ t\in \mathbb{R}_{>0} : t> \sqrt{3}m/2\}$
is contained in a chamber of
$\mathbb{R}_{>0} \setminus \sS$. 
Next we see that the tilt semistable objects in
this chamber coincide with slope semistable sheaves
in $\Coh_{\le 2}(X)$. 
\begin{prop}\label{prop:larget}
For $t>\sqrt{3}m/2$, an object
$E \in \bB_{B, H}$ 
with $\ch(E)=v$
is $\nu_{t}$-semistable 
if and only if $E$ is 
an $\mu_{H, 2}$-semistable sheaf
in $\Coh_{\le 2}(X)$. 
\end{prop} 
\begin{proof}
First suppose that $E\in \bB_{B, H}$
with $\ch(E)=v$ is 
$\nu_t$-semistable for some $t>\sqrt{3}m/2$. 
By Lemma~\ref{wall:S}, this is equivalent to the 
fact that 
$E$ is $\nu_t$-semistable for any $t>\sqrt{3}m/2$.  
We have the exact sequence in 
$\bB_{B, H}$, 
\begin{align*}
0 \to \hH^{-1}(E)[1] \to E \to \hH^0(E) \to 0. 
\end{align*}
Suppose that $\hH^{-1}(E) \neq 0$. 
Then we have 
\begin{align*}
\nu_t(\hH^{-1}(E)[1])>0
\end{align*}
for $t\gg 0$, which contradicts to the fact
that $E$ is 
$\nu_t$-semistable for any $t>\sqrt{3}m/2$. 
Hence $\hH^{-1}(E)=0$
and $E$ is 
a torsion sheaf, i.e. $E\in \Coh_{\le 2}(X)$. 
Noting that $\nu_{t}=\mu_{B, tH, 2}$ on $\Coh_{\le 2}(X)$
and $B$, $tH$ are proportional to $H$, 
we conclude that $E$ is a $\mu_{H, 2}$-semistable sheaf in 
$\Coh_{\le 2}(X)$. 

Conversely, take a $\mu_{H, 2}$-semistable sheaf $E \in \Coh_{\le 2}(X)$
with $\ch(E)=v$.  
Note that $E \in \bB_{B, H}$. Suppose that $E$ is not $\nu_t$-semistable 
for some $t > \sqrt{3}m/2$.  
By Lemma~\ref{wall:S}, this is equivalent to the fact that 
$E$ is not $\nu_t$-semistable for any $t> \sqrt{3}m/2$, 
hence we may assume that $t>\sqrt{3}m$. 
There is an exact sequence in $\bB_{B, H}$, 
\begin{align}\label{E1E2}
0 \to E_1 \to E \to E_2 \to 0, 
\end{align}
such that $E_1$ is $\nu_t$-semistable 
and $\nu_t(E_1)>\nu_t(E_2)$. 
Let us write $\ch^{B}(E_1)=(r_1, D_1, \beta_1, n_1)$. 
Since $E\in \Coh_{\le 2}(X)$, we have 
$E_1 \in \Coh(X)$, hence $r_1 \ge 0$. 
By $\nu_t(E_1) \ge 0$ and Theorem~\ref{thm:class}, 
we obtain the inequalities, 
\begin{align*}
D_1^2 H \ge 2r_1 \beta_1 H \ge \frac{H^3}{3}t^2 r_1^2. 
\end{align*}
Suppose that $r_1>0$. 
Since $\ch^{B}_1(E)^2 H \ge D_1^2 H$, we obtain the 
inequality
$m^2 H^3 \ge H^3 t^2/3$, which contradicts to 
$t> \sqrt{3}m$. 
Therefore we have 
$r_1=0$, and the sequence (\ref{E1E2})
is an exact sequence in $\Coh_{\le 2}(X)$. 
However the $\mu_{H, 2}$-stability of 
$E$ implies 
$\nu_t(E_1) \le \nu_t(E_2)$, a contradiction. 
\end{proof}
As a corollary, we can show that
Conjecture~\ref{intro:conj} (i) is true 
under the assumption that Conjecture~\ref{conj:BMT}
is true: 
\begin{cor}\label{cor:conji}
Suppose that Conjecture~\ref{conj:BMT}
is true. Then if there is an $\mu_{H, 2}$-semistable 
sheaf $E \in \Coh_{\le 2}(X)$ with $\ch(E)=(0, mH, -\beta, -n)$, 
then we have 
\begin{align}\label{ineq:mH}
-\frac{H^3}{24}m^3 \le n+ \frac{(H\cdot \beta)^2}{2mH^3}.
\end{align}
In particular if $\DT_{H}(0, mH, -\beta, -n) \neq 0$, 
then the inequality (\ref{ineq:mH}) is satisfied.  
\end{cor}
\begin{proof}
Let $E$ be an $\mu_{H, 2}$-semistable 
object with $\ch(E)=(0, mH, -\beta, -n)$. 
By Proposition~\ref{prop:larget},
$E$ is $\nu_{\sqrt{3}m/2}$-semistable 
with $\nu_{\sqrt{3}m/2}(E)=0$. 
If we assume Conjecture~\ref{conj:BMT}, then 
we have 
\begin{align*}
\frac{H^3}{24}m^3(1-\eta) \le \frac{1}{18}\cdot mH \cdot \left( \frac{\sqrt{3}}{2}m \right)^2 = \frac{H^3}{24}m^3. 
\end{align*}
Here $\eta$ is given by (\ref{def:eta}). 
The above inequality implies $\eta \ge 0$, 
which is equivalent to the inequality (\ref{ineq:mH}). 

If $\DT_H(0, mH, -\beta, -n) \neq 0$, then there is 
an $H$-semistable sheaf $E \in \Coh(X)$ with 
$\ch(E)=(0, mH, -\beta, -n)$. 
By Remark~\ref{rmk:check}, $E$ is also an 
$\mu_{H, 2}$-semistable sheaf in $\Coh_{\le 2}(X)$, 
hence the inequality (\ref{ineq:mH}) holds. 
\end{proof}
\begin{rmk}
The inequality (\ref{ineq:mH}) is easy to prove 
if there is an $\mu_{H, 2}$-semistable sheaf 
$E$ with $\ch(E)=(0, mH, -\beta, -n)$
supported on a smooth member $P \in \lvert mH \rvert$. 
Indeed, such a sheaf is written as (\ref{fo:form}), and 
the inequality (\ref{ineq:mH}) can be easily checked 
using the Hodge index theorem. However when the support of 
$E$ is singular, we are not 
able to prove (\ref{ineq:mH}) without assuming 
Conjecture~\ref{conj:BMT}. 
As for the issue of the thickening of the support
of $E$, see Theorem~\ref{thm:append}. 
\end{rmk}
Later we will also use the following corollary. 
We use the notation in the previous subsection.  
\begin{cor}\label{cor:mustable}
Take $(\xi, \mu) \in \mathbb{R}^2$
so that $\xi>1$, $\mu>0$ or $\xi=1$, $0<\mu<3/2$. 
Suppose that Conjecture~\ref{conj:BMT} is true. 
Then there is $m(\xi, \mu)>0$, which depends only on 
$\xi$, $\mu$ such that if $m>m(\xi, \mu)$ and 
$0\le \eta <\mu/m^{\xi}$, then any $\mu_{H, 2}$-semistable 
sheaf $E \in \Coh_{\le 2}(X)$ with $\ch(E)=v$ is 
$\mu_{H, 2}$-stable. 
\end{cor}
\begin{proof}
For an $\mu_{H, 2}$-semistable sheaf $E$ with $\ch(E)=v$, 
suppose that $E$ is not $\mu_{H, 2}$-stable. Then 
there is an exact sequence in $\Coh_{\le 2}(X)$
\begin{align*}
0 \to E_1 \to E \to E_2 \to 0
\end{align*}
such that $\mu_{H, 2}(E_1)=\mu_{H, 2}(E_2)$. 
This is equivalent to 
$\mu_{B, H, 2}(E_1)=\mu_{B, H, 2}(E_2)=0$, where 
$B$ is given by (\ref{B:choice}).
Hence we can write 
\begin{align*}
\ch^{B}(E_i)=\left(0, m_i H, 0, \frac{H^3}{24}m_i^3(1-\eta_i)   \right)
\end{align*} 
for some $m_i \in \mathbb{Z}_{\ge 1}$, $\eta_i \in \mathbb{Q}$
satisfying $m_1 +m_2=m$ and 
\begin{align}\label{mm1m2}
\frac{H^3}{24}m^3(1-\eta)
=\frac{H^3}{24}m_1^3(1-\eta_1)
+\frac{H^3}{24}m_2^3(1-\eta_2). 
\end{align}
Also note that $\eta_i \ge 0$ by Corollary~\ref{cor:conji}. 
On the other hand, we have 
\begin{align}\notag
&m^3(1-\eta)-m_1^3(1-\eta_1)-m_2^3(1-\eta_2) \\
\label{ineq:m1}
&\ge m^3-m^3 \eta -m(m^2-3m_1m_2) \\
\label{ineq:m2}
& \ge 3m^2 -3m -m^3 \eta \\
\label{ineq:m3}
& > 3m^2 -\mu m^{3-\xi} -3m.  
\end{align}
Here we have used $\eta_i \ge 0$ in (\ref{ineq:m1})
and 
\begin{align*}
m_1^3 + m_2^3 = m(m^2 -3m_1 m_2)
\end{align*}
by $m_1 + m_2=m$. 
The inequality (\ref{ineq:m2})
follows from
$m_1 m_2 \ge m-1$ 
since $m_1, m_2 \ge 1$, 
and (\ref{ineq:m3}) follows from 
$\eta<\mu/m^{\xi}$. 
By our choice of $(\xi, \mu)$, the leading coefficient 
of the RHS of (\ref{ineq:m3}) is positive. 
Hence there is $m(\xi, \mu)>0$, which depends only on 
$\mu$ and $\xi$ such that  $m>m(\xi, \mu)$ implies 
(\ref{ineq:m3}) is positive. This contradicts (\ref{mm1m2}), 
hence $E$ is $\mu_{H, 2}$-stable.   
\end{proof}
Finally in this subsection, we see that 
there are no tilt semistable objects 
when $\eta<1$ and $t$ is small. 
\begin{lem}\label{lem:tsmall}
Suppose that Conjecture~\ref{conj:BMT} is true, 
and $\eta$ satisfies $\eta<1$. 
Then
there is no $\nu_t$-semistable object
$E\in \bB_{B, H}$ with $\ch(E)=v$, if 
\begin{align}\label{t:small}
t<\frac{\sqrt{3}}{2}m \sqrt{1-\eta}. 
\end{align}
\end{lem}
\begin{proof}
Let $E \in \bB_{B, H}$ be an $\nu_t$-semistable object with 
$\ch(E)=v$. If Conjecture~\ref{conj:BMT} is true, then 
\begin{align*}
\frac{H^3}{24}m^3(1-\eta) \le \frac{H^3}{18}mt^2. 
\end{align*}
The above inequality is violated under the condition
(\ref{t:small}).
\end{proof}

\subsection{Moduli stacks of tilt semistable objects}
Similarly to the moduli theory of semistable sheaves, 
we can consider moduli stacks of tilt semistable objects
in $\bB_{B, H}$. 
Let $\mM$ be the stack of all the objects
in $\bB_{B, H}$.
An argument similar to~\cite[Lemma~4.7]{Tst3}
shows that $\mM$ is an Artin 
stack locally of finite type over $\mathbb{C}$. 
Moreover the stack $\mM$ is an open substack 
of the moduli stack of certain
 objects in $D^b \Coh(X)$
constructed by Inaba~\cite{Inaba}
or Lieblich~\cite{LIE}. 
For each $w \in H^{\ast}(X, \mathbb{Q})$, 
there is an abstract substack
\begin{align}\label{M:abstract}
\mM_t^{ss}(w) \subset \mM
\end{align}
which is the moduli stack of $\nu_t$-semistable 
objects $E\in \bB_{B, H}$ with $\ch(E)=w$. 
We need to prove that $\mM_t^{ss}(w)$
is a `good' moduli space. 
We have the following proposition:
\begin{prop}\label{prop:moduli}
The stack $\mM_t^{ss}(w)$ is an open substack of 
$\mM$, and it is an Artin 
stack of finite type over $\mathbb{C}$. 
\end{prop}
\begin{proof}
A similar result for K3 surfaces
 is already obtained in~\cite{Tst3}, 
and we apply a similar argument. 
A proof similar to~\cite[Theorem~3.20]{Tst3}
shows that the problem can be reduced to 
showing the boundedness of $\nu_t$-semistable
objects $E\in \bB_{B, H}$ with $\ch(E)=w$. 
Furthermore if $\nu_t(w)=\infty$, then 
$E$ is 
contained in the following category, 
(cf.~\cite[Remark~3.2.2]{BMT},) 
\begin{align*}
\langle \uU[1], \Coh_{\le 1}(X) : 
\uU \mbox{ is } \mu_{B, H} \mbox{-stable with }
\mu_{B, H}(E)=0 \rangle. 
\end{align*}
In this case, the boundedness follows from
an argument similar to~\cite[Proposition~3.13]{Tolim}.
Below, we assume $\nu_t(w)<\infty$.  

For an $\nu_t$-semistable object
$E\in \bB_{B, H}$ with $\ch(E)=w$, 
we 
consider the filtration
\begin{align*}
F \subset T \subset \hH^0(E)
\end{align*}
where $T$ is the torsion part of $\hH^0(E)$
and $F$ is the maximal subsheaf of $T$ 
which is contained in $\Coh_{\le 1}(X)$. 
It is enough to show the boundedness of 
$\hH^{-1}(E)$, $F$, $T/F$ and $\hH^0(E)/T$. 

First we check the boundedness of $\hH^{-1}(E)$. 
Since $\nu_t(E)<\infty$, 
we have $\Hom(\Coh_{\le 1}(X), \hH^{-1}(E)[1])=0$. 
This immediately implies that $\hH^{-1}(E)$ is a reflexive 
sheaf. Also an argument similar to~\cite[Proposition~4.11]{Tst3}
shows that $\ch_i(\hH^{-1}(E))$ for $i=0, 1, 2$
have only a finite number of possibilities. 
Then applying~\cite[Theorem~4.4]{Langer}, 
the boundedness of $\hH^{-1}(E)$ follows.

Next we check the boundedness of $F$, $T/F$ and 
$\hH^0(E)/T$. 
Again, an argument similar to~\cite[Proposition~4.11]{Tst3}
shows that, for $i=0, 1, 2$, 
 $\ch_i(F)$, $\ch_i(T/F)$ and 
$\ch_i(\hH^0(E)/T)$ have only a finite number of possibilities. 
In order to apply~\cite[Theorem~4.4]{Langer}, 
we need to check that $\ch_3(F)$, $\ch_3(T/F)$
and $\ch_3(\hH^0(E)/T)$ have also a finite number of 
possibilities. 
Since the sum of them are equal to 
$\ch_3(\hH^{-1}(E)) + \int_{X} w$, it is enough to show that 
they are bounded above. 

As for the upper bound on $\ch_3(\hH^0(E)/T)$, we consider the exact sequence 
\begin{align*}
0 \to \hH^0(E)/T \to \left( \hH^0(E)/T \right)^{\vee \vee} \to F' \to 0
\end{align*}
where $F' \in \Coh_{\le 1}(X)$. 
Since the middle sheaf is reflexive, we can apply~\cite[Theorem~3.4]{Langer2}
and~\cite[Theorem~4.4]{Langer} to show that 
the middle sheaf is contained in a bounded family. 
Then by the boundedness of the Quot scheme, 
$\ch_3(F')$ can be shown to be bounded below, 
using an argument of~\cite[Lemma~3.10]{Tolim}. 
Therefore $\ch_3(\hH^0(E)/T)$ is bounded above. 
The upper bound on $\ch_3(T/F)$ is similarly obtained 
by taking $\eE xt_{X}^{1}(-, \oO_X)$
twice instead of the double dual, and applying~\cite[Theorem~4.4]{Langer}
again. 

As for the upper bound on $\ch_3(F)$, note that 
$E \in \bB_{B, H}$ has a subobject $E' \in \bB_{B, H}$ 
which fits into an exact sequence in $\bB_{B, H}$ 
\begin{align*}
0 \to \hH^{-1}(E)[1] \to E' \to F \to 0. 
\end{align*}
Since $\Hom(\Coh_{\le 1}(X), E)=0$, 
we have $\Hom(\Coh_{\le 1}(X), E')=0$. 
Then the upper bound on $\ch_3(F)$
is obtained by Lemma~\ref{lem:ch3} below. 
\end{proof}

\begin{rmk}
Although we assume that $X$ is a smooth projective 
Calabi-Yau 3-fold in this section, 
the result of Proposition~\ref{prop:moduli}
holds for any smooth projective 3-fold. 
\end{rmk}

We have used the following lemma:
\begin{lem}\label{lem:ch3}
For a fixed reflexive sheaf
 $U$ on $X$, consider the set of 
sheaves $F \in \Coh_{\le 1}(X)$
with fixed $\ch_2(F)$  
which fit into a distinguished triangle
\begin{align}\label{UVF}
U[1] \to V \to F
\end{align}
such that $\Hom(\Coh_{\le 1}(X), V)=0$. 
Then $\ch_3(F)$ is bounded above. 
\end{lem}
\begin{proof}
Since $U$ is a reflexive sheaf, 
we have $\eE xt_X^0(U, \oO_X)=U^{\vee}$, 
$Q \cneq \eE xt_X^1(U, \oO_X)$ is a zero dimensional sheaf
and $\eE xt_X^i(U, \oO_X)=0$ for $i\neq 0, 1$. 
Also $F^{\vee} \cneq \eE xt_X^2(F, \oO_X)$
is a pure one dimensional sheaf, 
$Q' \cneq \eE xt_X^3(F, \oO_X)$ is a zero dimensional 
sheaf and $\eE xt_X^i(F, \oO_X)=0$ for 
$i\neq 2, 3$. 
Applying $\mathbb{D}(-) \cneq \dR \hH om_X(-, \oO_X)$
to the triangle (\ref{UVF}), we obtain
$\hH^i(\mathbb{D}(V))=0$ for 
$i\neq 1, 2, 3$ and
the long exact sequence of sheaves, 
\begin{align*}
0 \to \hH^1(\mathbb{D}(V)) \to 
U^{\vee} &\to F^{\vee} \to \hH^2(\mathbb{D}(V)) \\
&\to Q \to Q' \to \hH^3(\mathbb{D}(V)) \to 0.
\end{align*}
Also the condition $\Hom(\Coh_{\le 1}(X), V)=0$ 
implies that
\begin{align}\label{cond:dual}
\Hom(\mathbb{D}(V), A[-2])=0, \ 
\Hom(\mathbb{D}(V), \oO_x[-3])=0
\end{align}
for any pure one dimensional sheaf $A$
and a closed point $x\in X$. 
Therefore $\hH^3(\mathbb{D}(V))=0$, 
$Q \to Q'$ is surjective
and $U^{\vee} \to F^{\vee}$ is surjective 
in dimension one. 
This implies that $\ch_3(Q')$ is bounded, 
and
an argument similar to~\cite[Lemma~3.10]{Tolim}
shows that $\ch_3(F^{\vee})$ is bounded below. 
Hence $\ch_3(\mathbb{D}(F))$ is bounded below, 
and $\ch_3(F)$ is bounded above. 
\end{proof}
By Remark~\ref{rmk:check},
Proposition~\ref{prop:larget}, Corollary~\ref{cor:mustable}
and Lemma~\ref{lem:tsmall}, the following lemma
obviously follows: 
\begin{lem}\label{lem:obvious}
Suppose that Conjecture~\ref{conj:BMT} is true, 
and take $\xi$, $\mu$ and $m(\xi, \mu)$
as in Corollary~\ref{cor:mustable}. 
If $m>m(\xi, \mu)$ and $\eta<\mu/m^{\xi}$,
we have the following: 

(i) For $t>\sqrt{3}m/2$, we have 
\begin{align*}
\mM_{t}^{ss}(0, mH, -\beta, -n)
= \mM_{H}^{s}(0, mH, -\beta, -n). 
\end{align*}

(ii) For $t< \sqrt{3}m/2 \cdot \sqrt{1-\eta}$, 
we have 
\begin{align*}
\mM_t^{ss}(0, mH, -\beta, -n) = \emptyset.  
\end{align*}
\end{lem}

\subsection{Wall-crossing in the Hall algebra}
As we discussed in Subsection~\ref{subsec:wall}, there 
is the set of walls $\sS \subset \mathbb{R}_{>0}$
such that the moduli stack (\ref{M:abstract})
may jump if we cross the wall. 
We investigate the behavior 
of the moduli stack (\ref{M:abstract})
under the change of $t$ in terms of 
Hall algebra. 

Recall that the Hall algebra
$H(\bB_{B, H})$ of $\bB_{B, H}$ 
is spanned over $\mathbb{Q}$ by the isomorphism classes of 
symbols, 
\begin{align*}
[\rho \colon \xX \to \mM], 
\end{align*}
where $\xX$ is an Artin stack of finite type over $\mathbb{C}$
with affine geometric stabilizers and 
$\rho$ is a 1-morphism. 
The relations are generated by
\begin{align*}
[\rho \colon \xX \to \mM] \sim
[\rho|_{\yY} \colon \yY \to \mM] 
+ [\rho|_{\uU} \colon \uU \to \mM] 
\end{align*} 
where $\yY \subset \xX$ is a closed substack 
and $\uU \cneq \xX \setminus \yY$.
 
There is an associative $\ast$-product
on the $\mathbb{Q}$-vector space $H(\bB_{B, H})$
based on Ringel-Hall algebras. 
Let $\eE x$  be the stack of 
short exact sequences in $\bB_{B, H}$, 
\begin{align*}
0 \to E_1 \to E_2 \to E_3 \to 0.
\end{align*} 
There are 1-morphisms, 
\begin{align*}
p_i \colon \eE x  \to \mM, \quad 
i=1, 2, 3
\end{align*}
sending $E_{\bullet}$ 
to $E_i$. 
The $\ast$-product
is defined by  
\begin{align*}
\ast \cneq p_{2\ast}(p_1, p_3)^{\ast}
\colon H(\bB_{B, H})^{\otimes 2} \to H(\bB_{B, H}). 
\end{align*}
For the detail, see~\cite{Joy2}.

For each $t\in \mathbb{R}_{>0}$
and $w \in H^{\ast}(X, \mathbb{Q})$, 
the stack (\ref{M:abstract})
determines an element by Proposition~\ref{prop:moduli}, 
\begin{align*}
\delta_{t}(w) \cneq 
[\mM_{t}^{ss}(w) \subset
\mM] \in H(\bB_{B, H}). 
\end{align*}
We take $v\in H^{\ast}(X, \mathbb{Q})$
as in (\ref{v:class}). 
Then for each $t_0\in \mathbb{R}_{>0}$, 
the existence of Harder-Narasimhan 
filtrations with respect to the tilt stability 
yields, 
\begin{align}\label{delta:sum}
\delta_{t_0}(v)=\sum_{l\ge 1}
\sum_{\begin{subarray}{c}
v_1 + \cdots + v_l=v, \\
\nu_{t_0}(v_1)= \cdots =\nu_{t_0}(v_l)=0, \\
\nu_{t_{+}}(v_1)> \cdots > \nu_{t_{+}}(v_l)
\end{subarray}}
\delta_{t_{+}}(v_1)\ast \cdots \ast \delta_{t_{+}}(v_l). 
\end{align}
Here $t_{+}=t_0+\varepsilon$
for $0<\varepsilon \ll 1$. 
The above formula is a starting point 
to deduce the wall-crossing formula in~\cite{Joy4}, \cite{JS} and
\cite{K-S}. 

Below we investigate the RHS of (\ref{delta:sum})
assuming that Conjecture~\ref{conj:BMT} is true, 
and $\eta$ is sufficiently small, 
where $\eta$ is given by (\ref{def:eta}). 
We use the following condition: 
\begin{align}\label{cond:eta}
0 \le \eta < \mathrm{min} \left\{ \frac{3}{2m}, \frac{1}{2} \right\}.
\end{align}
Note that, by Corollary~\ref{cor:conji}, 
 the left inequality 
automatically follows if we assume Conjecture~\ref{conj:BMT}.
We have the following proposition: 
\begin{prop}\label{RHS:delta}
Suppose that Conjecture~\ref{conj:BMT}
is true. 
Assume that $\eta$ satisfies (\ref{cond:eta}). 
 We have the following:

(i) A non-zero term of the RHS of (\ref{delta:sum})
satisfies $l=1$ or $l=2$. 

(ii) For a non-zero term in the RHS of (\ref{delta:sum})
with $l=2$, 
the numerical classes 
$v_1$, $v_2$ satisfy $e^{-B}v_i=
(r_i, D_i, \gamma_i, s_i) \in 
H^{\ast}(X, \mathbb{Q})$
with $r_1=-1$ and $r_2=1$.  

(iii) In the notation of (ii), 
we write $D_i=d_i H$
for $d_i \in \mathbb{Q}$. Then
we have 
\begin{align}\label{eval:d}
&\left\lvert d_i -\frac{m}{2} \right\rvert < \frac{1}{3}m\eta \\
\label{eval:g}
\frac{H^3}{8}m^2 \left(1-\frac{2}{3}\eta \right)^2
&< (-1)^i \gamma_i H 
< \frac{H^3}{8}m^2 \left(1+ \frac{2}{3}\eta \right)^2. 
\end{align}
\end{prop}
\begin{proof}
It is enough to consider the terms with $l\ge 2$. 
Let $E\in \bB_{B, H}$ be an $\nu_{t_0}$-semistable 
object with $\ch(E)=v$
 which is not $\nu_{t_{+}}$-semistable. 
Then there is a filtration in $\bB_{B, H}$
\begin{align*}
E_1 \subset E_2 \subset \cdots \subset E_l =E
\end{align*}
such that each subquotient
$F_i \cneq E_i/E_{i-1}$ is $\nu_{t_{+}}$-semistable 
with $\nu_{t_0}(F_i)=0$ and 
\begin{align}\label{ineq:F}
\nu_{t_{+}}(F_1)> \nu_{t_{+}}(F_2)> \cdots > \nu_{t_{+}}(F_l). 
\end{align}
Let us write $\ch^{B}(F_i)=(r_i, D_i, \gamma_i, s_i) 
\in H^{\ast}(X, \mathbb{Q})$. 
Then the condition 
$\nu_{t_0}(F_i)=0$ together with 
the inequalities (\ref{ineq:F})
imply
\begin{align*}
\frac{r_1}{D_1 H^2} < \frac{r_2}{D_2 H^2}< \cdots 
< \frac{r_l}{D_l H^2}. 
\end{align*}
Since $r_1 + \cdots + r_l=0$, there is 
$1\le i \le l-1$ such that 
$r_j<0$ for $j\le i$, $r_{i+1} \ge 0$ and $r_j >0$ for 
$j\ge i+2$. 
We set $S \cneq E_{i}$, $T \cneq E/E_{i}$, 
and write 
$\ch^{B}(S)=(r_S, D_S, \gamma_S, s_S)$, 
$\ch^{B}(T)=(r_T, D_T, \gamma_T, s_T)$. 
Since $S$, $T$
are $\nu_{t_0}$-semistable, 
for $\ast \in \{S, T\}$, we have 
the inequality by Theorem~\ref{thm:class}, 
\begin{align*}
D_{\ast}^2 H \ge 2 r_{\ast} \gamma_{\ast} H. 
\end{align*}
Also from $\nu_{t_0}(\ast)=0$, we have 
\begin{align*}
-\frac{H^3}{6}r_{\ast}t_0^2 + \gamma_{\ast}H=0. 
\end{align*}
Since we assume Conjecture~\ref{conj:BMT}, Lemma~\ref{lem:tsmall}
implies
\begin{align}\label{t2}
t_0^2 \ge \frac{3}{4}m^2(1-\eta). 
\end{align}
Since $D_S +D_T =mH$, 
either $\ast=S$ or $T$ satisfies
$m^2 H^3/4 \ge D_{\ast}^2 H$. 
Hence for such $\ast$, we have
\begin{align}\label{ineq:ast}
\frac{H^3}{4}m^2 \ge 
D_{\ast}^2 H \ge
2r_{\ast} \cdot \frac{H^3}{6}r_{\ast}t_0^2
\ge \frac{H^3}{4}r_{\ast}^2 m^2(1-\eta). 
\end{align}
Therefore if $\lvert r_{\ast} \rvert \ge 2$,
the above inequalities imply
$\eta \ge 3/4$, which contradicts to
the assumption (\ref{cond:eta}).  
Since $r_S +r_{T}=0$, we conclude that
$r_{S}=-1$ and $r_{T}=1$. 
It follows that there are two possibilities:
$l=2$, $r_1=-1$, $r_2=1$
and $l=3$, $r_1=-1$, $r_2=0$, $r_3=1$. 

Let us exclude the latter case. Suppose that $l=3$, 
and we write $D_i=d_i H$ for $1\le i\le 3$
with $d_i \in \mathbb{Q}$. 
Note that $d_1 + d_2 + d_3=m$, 
$d_i>0$
and $d_2 \in \mathbb{Z}_{\ge 1}$
 since $r_2=0$. 
Hence there is $i \in \{1, 3\}$ such that 
$d_i \le m/2$. 
For such $i$,  
since $\lvert r_i \rvert =1$, 
we have the inequalities similar to 
(\ref{ineq:ast}),
\begin{align*}
\frac{H^3}{4}m^2 \ge d_i^{2} H^3 \ge \frac{H^3}{4}m^2(1-\eta), 
\end{align*}
which imply 
\begin{align*}
\frac{m}{2} \left( 1-\frac{2}{3}\eta \right)
< \frac{m}{2} \sqrt{1-\eta}
\le d_i \le \frac{m}{2}. 
\end{align*} 
Here the left inequality follows from 
(\ref{cond:eta}). 
Therefore we have 
\begin{align}\label{ineq:ast13}
\left \lvert d_i -\frac{m}{2} \right\rvert
< \frac{1}{3}m\eta, \ i=1, 3, 
\end{align}
and $\lvert d_2 \rvert <2m\eta/3$.
Combined with the assumption (\ref{cond:eta}), 
we have $\lvert d_2 \rvert <1$, 
which contradicts to $d_2 \in \mathbb{Z}_{\ge 1}$. 
Therefore the case $l=3$ is excluded,
and (i), (ii) are proved. 
When $l=2$, the argument showing
the inequality (\ref{ineq:ast13}) also shows (\ref{eval:d}). 

Finally we prove (\ref{eval:g}). 
For simplicity, we only show the case $i=2$. 
By $\nu_{t_0}(v_2)=0$, (\ref{cond:eta}) and (\ref{t2}), 
we have  
\begin{align*}
\gamma_2 H = \frac{H^3}{6}t_0^2 \ge \frac{H^3}{8}m^2(1-\eta)
> \frac{H^3}{8}m^2 \left( 1-\frac{2}{3}\eta \right)^2. 
\end{align*}
On the other hand, by (\ref{eval:d})
and using Theorem~\ref{thm:class}, 
we have 
\begin{align*}
\gamma_2 H \le \frac{1}{2}d_2^2 H^3 <
\frac{H^3}{8}m^2 \left( 1+\frac{2}{3}\eta \right)^2. 
\end{align*}
Hence the inequalities (\ref{eval:g}) hold. 
\end{proof}

\subsection{Rank $\pm 1$ tilt semistable objects}
As we observed in Proposition~\ref{RHS:delta}, 
only rank $\pm 1$ tilt semistable
 objects are involved in 
non-trivial wall-crossing factors of 
(\ref{delta:sum}). In this subsection, we see that
these objects are isomorphic to
ideal sheaves of curves and derived dual 
of stable pairs, up to tensoring 
line bundles and shift. 

We first characterize the derived dual 
of stable pairs. 
Below for a stable pair $(\oO_X \to F)$, 
we regard it as a two term complex
with $\oO_X$ located in degree zero.
Also the derived dualizing functor $\mathbb{D}$
is defined by 
\begin{align*}
\mathbb{D}(-) \cneq \dR \hH om_{X}(-, \oO_X). 
\end{align*} 
\begin{lem}\label{PT:dual}
For an object $E \in D^b \Coh(X)$, 
there are a stable pair $(\oO_X \to F)$, 
$\lL \in \Pic(X)$ and an isomorphism 
\begin{align}\label{E:pair}
E \cong \lL \otimes \mathbb{D}(\oO_X \to F)[1]
\end{align}
if and only if $\hH^{-1}(E) \in \Pic(X)$, 
$\hH^{0}(E) \in \Coh_{\le 1}(X)$, 
$\hH^i(E)=0$ for $i\neq -1, 0$ and 
$\Hom(\Coh_{\le 1}(X), E)=0$. 
\end{lem}
\begin{proof}
First we check the `only if' part. Suppose
that $E$ is written as (\ref{E:pair}). 
Applying $\otimes \lL \circ [1] \circ\mathbb{D}$
to the distinguished triangle
\begin{align*}
F[-1] \to (\oO_X \to F) \to \oO_X
\end{align*}
we obtain the distinguished triangle, 
\begin{align*}
\lL[1] \to E \to \eE xt^2_X(F, \oO_X) \otimes \lL.
\end{align*}
The conditions $\hH^{-1}(E) \in \Pic(X)$, 
$\hH^0(E) \in \Coh_{\le 1}(X)$ and $\hH^i(E)=0$
for $i\neq -1, 0$
follow from the above triangle. The
rest condition $\Hom(\Coh_{\le 1}(X), E)=0$
also follows using (\ref{cond:dual}), replacing $V$ by $E$. 

Next we check the `if' part. 
We set $\lL =\hH^{-1}(E) \in \Pic(X)$,
and $E'=\mathbb{D}(E \otimes \lL^{-1}[-1])$. 
We apply the functor $\mathbb{D}
\circ [-1]\circ \otimes \lL^{-1}$ to the 
distinguished triangle, 
\begin{align*}
\hH^{-1}(E)[1] \to E \to \hH^0(E)
\end{align*}
and apply the same argument in the proof of Lemma~\ref{lem:ch3}.
Then we see that $\hH^1(E')$ is zero dimensional
and $\hH^0(E') \cong I_C$ for some 
subscheme $C \subset X$ with $\dim C\le 1$. 
By~\cite[Lemma~3.11]{Tcurve1}, 
the object $E'$ is isomorphic to a two 
term complex determined by a stable pair. 
\end{proof}
Next we investigate the 
set of $\nu_t$-semistable 
objects for $t\gg 0$ with 
rank $\pm 1$. 
\begin{lem}\label{lem:rankone}
(i) For $w=(1, m'H, \gamma', s') \in H^{\ast}(X, \mathbb{Q})$, 
there is $t_{\rm{DT}}>0$
which depends on $w$
 such that 
an object $E\in \bB_{B, H}$
with $\ch(E)=w$ is 
$\nu_t$-semistable for $t>t_{\rm{DT}}$
 if and only if 
\begin{align}\label{DT:large}
E \cong \oO_X(m'H) \otimes I_{C}
\end{align}
for some $C\subset X$ with $\dim C \le 1$. 

(ii) For $w=(-1, m'H, \gamma', s') \in H^{\ast}(X, \mathbb{Q})$, 
there is $t_{\rm{PT}}>0$
which depends on $w$ such that 
an object $E\in \bB_{B, H}$
with $\ch(E)=w$ is $\nu_t$-semistable 
for $t>t_{\rm{PT}}$ if and only if 
\begin{align}\label{PT:large}
E \cong \oO_X(m'H) \otimes \mathbb{D}(\oO_X \to F)[1]
\end{align}
for some stable pair $(\oO_X \to F)$ on $X$. 
\end{lem}
\begin{proof}
The proofs of (i) and (ii) are similar, so 
we only prove (ii). 
By~\cite[Lemma~7.2.1]{BMT}, there is 
$t_{\rm{PT}}>0$ such that if $t>t_{\rm{PT}}$,
any $\nu_t$-semistable object $E\in \bB_{B, H}$
is such that $\hH^{-1}(E)$ is a torsion free 
$\mu_H$-semistable sheaf, $\hH^0(E) \in \Coh_{\le 1}(X)$
and $\Hom(\Coh_{\le 1}(X), E)=0$. 
The last condition also implies that $\hH^{-1}(E)$ is 
reflexive. Therefore 
$\hH^{-1}(E)$ is a rank one reflexive sheaf, which 
is a line bundle on $X$. 
Then $E$ is written as (\ref{PT:large}) by Lemma~\ref{PT:dual}. 

Conversely, suppose that $E$ is written as (\ref{PT:large})
and $E$ is not $\nu_t$-semistable. 
Since $\ch^{B}_0(E)=-1$, we may assume that $\nu_t(E)>0$. 
There is an exact sequence in $\bB_{B, H}$
\begin{align}\label{long:E}
0 \to E_1 \to E \to E_2 \to 0
\end{align}
such that $\nu_t(E_1) > \nu_t(E_2)$
and $E_1$ is $\nu_t$-semistable. 
We write $\ch^{B}(E_i)=(r_i, D_i, \gamma_i, n_i)$
for $i=1, 2$. 
By Lemma~\ref{PT:dual}, 
the long exact sequence of cohomologies
yields that $\hH^0(E_2) \in \Coh_{\le 1}(X)$, 
hence $r_2 \le 0$.  
Suppose that $r_1>0$. Then by Theorem~\ref{thm:class}
and $\nu_t(E_1)> 0$, we obtain 
\begin{align*}
\ch^{B}_1(E)^2 H \ge 
D_1^2 H \ge 2r_1 \gamma_1 H > \frac{H^3}{3}r_1^2 t^2
\ge \frac{H^3}{3} t^2. 
\end{align*}
The above inequalities are violated if 
$t>\sqrt{3}m$. 
Therefore we may assume that $r_1 \le 0$. 
Since $r_1 +r_2=-1$, there are two 
possibilities: 
$(r_1, r_2)=(-1, 0)$ and $(0, -1)$. 
In the first case, we have $E_2 \in \Coh_{\le 1}(X)$
which implies $\nu_t(E_2)=\infty$. 
This contradicts to $\nu_t(E_1)>\nu_t(E_2)$. 
In the latter case, 
since we have 
\begin{align*}
\nu_t(E_1) = \frac{t(\gamma_1 H)}{D_1 H^2} \ge \nu_t(E)
\end{align*}
and $\nu_t(E)$ is a cubic polynomial in $t$ with 
positive leading term, 
it is enough to give an upper bound on $\gamma_1 H$. 
This is equivalent to giving a lower bound 
of $\gamma_2 H$. 
The long exact sequence of cohomologies associated with 
(\ref{long:E}) implies 
that $\hH^{-1}(E_2) \cong \oO_X(m''H)$
for $m'' >m'$, and since $E_2 \in \bB_{B, H}$, 
the integer $m''$ is also bounded above. 
This implies that $\ch^{B}_2(\hH^{-1}(E_2)[1]) H$ is 
bounded. Also since $\ch^{B}_2(\hH^0(E_2)) H \ge 0$, 
we obtain a lower bound on $\gamma_2 H$. 
\end{proof}
Using the result of the previous lemma, 
we show that the objects contributing to 
non-trivial terms of the RHS of 
(\ref{delta:sum}) are of the form (\ref{DT:large}) or (\ref{PT:large}). 
The following proposition corresponds to 
core dump exponent conjecture in~\cite{DM}:
\begin{prop}\label{prop:core}
Assume that Conjecture~\ref{conj:BMT} is true
and $\eta$ satisfies
(\ref{cond:eta}). 
Let $\delta_{t_{+}}(v_1)\ast \delta_{t_{+}}(v_2)$
be a non-zero term of the RHS of (\ref{delta:sum}). 
Then we have 

(i) $E \in \bB_{B, H}$ is $\nu_{t_{+}}$-semistable 
with $\ch(E)=v_2$ if and only if 
$E$ is written as (\ref{DT:large}). 

(ii) $E \in \bB_{B, H}$ is $\nu_{t_{+}}$-semistable
with $\ch(E)=v_1$ if and only if 
$E$ is written as (\ref{PT:large}). 
\end{prop}
\begin{proof}
The proofs of (i) and (ii) are similar, so 
we only prove (i). 
By Proposition~\ref{RHS:delta}, 
we can write $e^{-B}v_2=(1, d_2 H, \gamma_2, s_2)$. 
Hence by 
Lemma~\ref{lem:rankone}, it is enough 
to show that there is no wall 
on $\{ t\in \mathbb{R}_{>0} : t\ge t_0 \}$
with respect to the numerical class $v_2$. 
Suppose the contrary. 
Then there is an $\nu_{t_0}$-semistable 
object $E\in \bB_{B, H}$ with $\ch(E)=v_2$ and 
$t_1\ge t_0$ such that $E$ is $\nu_{t}$-semistable 
for $t\in [t_0, t_1]$ but 
not semistable for $t=t_1 +\varepsilon$
with $0<\varepsilon \ll 1$. 
Note that, since $\nu_{t_0}(E)=0$
and $t_1 \ge t_0$, we have 
$\nu_{t_1}(E) \le 0$. 
There is an exact sequence in $\bB_{B, H}$
\begin{align*}
0 \to S \to E \to T \to 0
\end{align*}
such that
\begin{align}\label{nuAB}
\nu_{t_1}(S)=\nu_{t_1}(T) \le 0, \ 
\nu_{t_1 +\varepsilon}(S)> \nu_{t_1 +\varepsilon}(T),
\end{align}
for $0<\varepsilon \ll 1$. 
Let us write $\ch^{B}(\ast)=(r_{\ast}, d_{\ast}H, \gamma_{\ast}, s_{\ast})$
for $\ast \in \{ S, T \}$. 
By (\ref{nuAB}), we have 
$r_S/d_S<r_T/d_T$,
hence $r_S \le 0$ and $r_T \ge 1$
since $r_S +r_T=1$. 
Also by (\ref{nuAB}), we have 
\begin{align*}
\frac{H^3}{6}\left(\frac{r_T}{d_T}-\frac{r_S}{d_S}  \right)
t_1^2
=\frac{\gamma_T H}{d_T}-\frac{\gamma_S H}{d_S}.
\end{align*}
Suppose that $r_S \le -1$, hence $r_{T} \ge 2$. 
Since $S$ and $T$ are $\nu_{t_1}$-semistable, 
Theorem~\ref{thm:class} together 
with the above equality imply
\begin{align*}
\frac{H^3}{6}\left(\frac{r_T}{d_T}-\frac{r_S}{d_S}  \right)
t_1^2 \le \frac{d_T H^3}{2r_T} - \frac{d_S H^3}{2r_S}, 
\end{align*}
which is equivalent to 
\begin{align}\label{t12}
t_1^2 \le \frac{3d_S d_T}{-r_S r_T}. 
\end{align}
Since $d_S, d_T>0$ and $d_S +d_T=d_2$, 
we have $d_S d_T \le d_2^2/4$. 
Therefore by Lemma~\ref{lem:tsmall}, 
(\ref{eval:d}), (\ref{t12})
and the assumption $r_S \le -1$, $r_T \ge 2$, 
we have 
\begin{align*}
\frac{3}{4}m^2(1-\eta) \le t_0^2
\le t_1^2 \le \frac{3}{32} m^2
\left(1+ \frac{2}{3}\eta  \right)^2,
\end{align*}
The above inequalities contradict to
(\ref{cond:eta}). 
Hence we have $r_S=0$ and $r_T=1$. 

Since $r_S=0$, we have $d_S \in \mathbb{Z}_{\ge 1}$, 
and (\ref{nuAB}) implies $\gamma_S H \le 0$. 
Since $\gamma_2=\gamma_S +\gamma_T$
and $d_2=d_S +d_T$, we have 
\begin{align*}
\frac{H^3}{8}m^2 \left( 1-\frac{2}{3}\eta \right)^2
< \gamma_2 H \le \gamma_T H \le \frac{H^3}{2}d_T^2
\le \frac{H^3}{2}d_2^2 <
\frac{H^3}{8}m^2 \left( 1+\frac{2}{3}\eta \right)^2.
\end{align*}
Here the first inequality follows from 
(\ref{eval:g}), the third inequality follows from 
Theorem~\ref{thm:class} and the last inequality 
follows from (\ref{eval:d}). 
By the above inequalities, we have
\begin{align*}
\frac{m}{2} \left(1-\frac{2}{3}\eta  \right)
< d_T < \frac{m}{2} \left(1+\frac{2}{3}\eta  \right).
\end{align*}
Combined with (\ref{eval:d}),
we obtain that $\lvert d_S \rvert < 2m\eta/3$. 
By the assumption (\ref{cond:eta}), 
we have $\lvert d_S \rvert <1$, which 
contradicts to $d_S \in \mathbb{Z}_{\ge 1}$.  
\end{proof}

\subsection{Numerical classes}
Let $\delta_{t_{+}}(v_1) \ast \delta_{t_{+}}(v_2)$
be a non-zero term of the RHS of (\ref{delta:sum}). 
By the results in the previous subsection, 
the numerical classes $v_1$, $v_2$ are 
written as 
\begin{align}\label{v12}
v_1 =-e^{m_1 H}(1, 0, -\beta_1, -n_1), \
v_2 =e^{m_2 H}(1, 0, -\beta_2, -n_2)
\end{align}
for $m_i \in \mathbb{Z}$, 
$\beta_i$ are Poincar\'e duals of 
homology classes of effective
algebraic one cycles on $X$, 
and $n_i \in \mathbb{Z}$. 
In this subsection, we bound 
$\beta_i$ and $n_i$. 

For $\epsilon>0$,
let $C(m, \epsilon) \subset H_2(X) \oplus H_0(X)$
be the subset defined by (\ref{Cme}). 
The following result corresponds to 
the extreme polar state conjecture in~\cite{DM}:
\begin{prop}\label{prop:extreme}
In the same assumptions 
as Proposition~\ref{prop:core}, 
we write $v_1$ and $v_2$ as (\ref{v12}). 
Then we have 
\begin{align}\notag
(\beta_i, n_i) \in C(m, \eta H^3/2). 
\end{align}
\end{prop}
\begin{proof}
We first describe the relationship between 
the notations of $v_i$ in Proposition~\ref{RHS:delta}
and (\ref{v12}). 
Let us set $b=-\frac{\beta H}{mH^3}$
so that $B=bH$ as in (\ref{B:choice}).
If we denote $e^{-B}v_i=(r_i, d_i H, \gamma_i, s_i)$
with $r_1=-1$, $r_2=1$ 
as in Proposition~\ref{RHS:delta}, 
we have 
$d_1=b-m_1$, $d_2=m_2-b$ and  
\begin{align*}
(\gamma_1, s_1) &=
\left(
\beta_1-\frac{1}{2}d_1^2 H^2, n_1-d_1 \beta_1 H +\frac{H^3}{6}d_1^3
\right) \\
(\gamma_2, s_2) &=
\left(
-\beta_2+\frac{1}{2}d_2^2 H^2, -n_2-d_2 \beta_2 H +\frac{H^3}{6}d_2^3
\right).
\end{align*}
By (\ref{eval:d}) and (\ref{eval:g}), 
 we can bound $\beta_2 H$ as follows: 
\begin{align*}
\beta_2 H &< \frac{H^3}{2}d_2^2 -\frac{H^3}{8}m^2 \left(1-\frac{2}{3}\eta
\right)^2 \\
&< \frac{H^3}{8}m^2 \left(1+\frac{2}{3}\eta
\right)^2 -\frac{H^3}{8}m^2 \left(1-\frac{2}{3}\eta
\right)^2  \\
&= \frac{H^3}{3}\eta m^2 < \frac{H^3}{2}\eta m^2. 
\end{align*}
A similar computation shows that 
\begin{align}\label{beta1H}
\beta_1 H < \frac{H^3}{3}\eta m^2
< \frac{H^3}{2}\eta m^2. 
\end{align}

Next we bound $\lvert n_i \rvert$. 
Since $\nu_{t_0}(v_1)=0$, we have 
\begin{align*}
\frac{H^3}{6}t_0^2 + \left( \beta_1-\frac{H^2}{2}d_1^2  \right)H=0.
\end{align*}
Since we assume Conjecture~\ref{conj:BMT}, we have 
\begin{align*}
n_1-d_1 \beta_1 H + \frac{H^3}{6}d_1^3 \le 
\frac{H^3}{18}d_1 t_0^2 
= \frac{d_1 H}{3} \left(-\beta_1 + \frac{H^3}{2}d_1^2  \right). 
\end{align*}
Therefore by (\ref{beta1H}) and noting 
$0<d_1<m$,  
we have 
\begin{align}\label{ueval:n1}
n_1 \le \frac{2}{3}d_1 \beta_1 H < \frac{2H^3}{9}\eta m^3< 
\frac{H^3}{2}\eta m^3. 
\end{align}
A similar computation 
also shows that 
\begin{align}\label{n2:similar}
n_2 > -\frac{2H^3}{9}\eta m^3 > -\frac{H^3}{2} \eta m^3. 
\end{align} 

In order to give a lower bound on $n_1$ and an upper bound 
of $n_2$, we bound $n_1-n_2$. 
Since $v_1 +v_2=v$ where $v$ is given by (\ref{v:class}), 
we have 
\begin{align*}
n_1 -n_2 -d_1 \beta_1 H -d_2 \beta_2 H +\frac{H^3}{6}(d_1^3 + d_2^3)
= \frac{H^3}{24}m^3(1-\eta). 
\end{align*}
Since $d_i >0$, $\beta_i H \ge 0$
and $d_1 + d_2=m$, the above equality implies
\begin{align}\label{ineq:n12}
n_1 -n_2 \ge \frac{H^3}{24}m^3(1-\eta)
 -\frac{H^3}{6}m(d_1 ^2 + d_2^2 -d_1 d_2). 
\end{align}
By (\ref{eval:d}), we can bound $d_1^2 + d_2^2 -d_1 d_2$
as 
\begin{align*}
d_1 ^2 + d_2^2 -d_1 d_2 &\le \frac{m^2}{2} \left(1+ \frac{2}{3}\eta \right)^2
-\frac{m^2}{4} \left(1-\frac{2}{3}\eta \right)^2 \\
&=\frac{m^2}{4} \left(\frac{4}{9}\eta^2 + 4\eta +1  \right).
\end{align*}
Combined with the inequality (\ref{ineq:n12}), we obtain 
\begin{align*}
n_1-n_2 &\ge \frac{H^3}{24}m^3(1-\eta)
-\frac{H^3}{24}m^3\left(\frac{4}{9}\eta^2 +4\eta +1 \right) \\
&=-\frac{5H^3}{24}m^3 \eta - \frac{H^3}{54}m^3 \eta^2. 
\end{align*}
Combined with (\ref{n2:similar}), 
we obtain 
\begin{align}\notag
n_1 &> -\frac{5H^3}{24}m^3 \eta - \frac{H^3}{54}m^3 \eta^2
-\frac{2H^3}{9}m^3 \eta \\
\label{leval:n1}
&= -\frac{31}{72}H^3 m^3 \eta - \frac{H^3}{54}m^3 \eta^2
> -\frac{H^3}{2}\eta m^3. 
\end{align}
By (\ref{ueval:n1}) 
and (\ref{leval:n1}), we have $\lvert n_1 \rvert < H^3 \eta m^3/2$. 
A similar computation also shows that
$\lvert n_2 \rvert < H^3 \eta m^3/2$. 
\end{proof}

We also use the following lemma: 
\begin{lem}\label{lem:revise:1}
Suppose that the numerical classes (\ref{v12})
satisfy 
\begin{align}\label{revise:1}
(\beta_i, n_i) \in C(m, \eta H^3/2), \quad 
v_1 + v_2=v. 
\end{align}

(i) We have $m_1<b < m_2$, 
where $B=bH$ in (\ref{B:choice}). 

(ii) There is only a finite number of 
possibilities for $v_1$ and $v_2$. 
\end{lem}
\begin{proof}
(i) 
The condition $v_1 +v_2=v$ implies that
$m_2-m_1=m$ and
\begin{align}\label{rev2}
\beta_2 - \beta_1 = \frac{H^2}{2} m(m_1 + m_2 -2b). 
\end{align}
By $(\beta_i, n_i) \in C(m, \eta H^3/2)$, 
$\beta_i H \ge 0$ 
and the condition $0 \le \eta <1$
in (\ref{cond:eta}),
we have 
\begin{align}\label{rev3}
\lvert 
\beta_1 H- \beta_2 H \rvert 
 < \frac{\eta H^3 m^2}{2} < \frac{H^3 m^2}{2}. 
\end{align}
By (\ref{rev2}) and (\ref{rev3}),  we obtain the inequality
\begin{align*}
\lvert m_1 + m_2 -2b \rvert <m. 
\end{align*}
The inequality $m_1<b<m_2$
follows from the above 
inequality and $m_2 -m_1=m$. 

(ii) The condition $(\beta_i, n_i) \in C(m, \eta H^3/2)$
and that $\beta_i$ are classes of 
effective algebraic one cycles on $X$
immediately imply that there is only a finite number
of possibilities for $\beta_i$ and $n_i$. 
It is enough to bound $m_1$ and $m_2$. 
By (i) and $m_2-m_1=m$, we have the 
inequalities
\begin{align*}
b-m < m_1 <b, \ b< m_2< b+m. 
\end{align*}
Therefore $m_i$ are also bounded. 

\end{proof}

\subsection{Wall-crossing of DT type invariants}
In this subsection, 
using the 
results of the previous subsections and the 
wall-crossing formula of DT type
invariants in~\cite{JS}, \cite{K-S}, 
we give a formula relating invariants (\ref{DT:0m}), (\ref{inv:I})
and (\ref{inv:P}). 

Suppose that Conjecture~\ref{conj:BMT} is true and 
$\eta$ satisfies (\ref{cond:eta}). 
Let $\delta_{t_{+}}(v_1) \ast \delta_{t_{+}}(v_2)$
be a non-zero term of the RHS of (\ref{delta:sum}), 
and write $v_i$ as (\ref{v12}). 
By Proposition~\ref{prop:core}, $\delta_{t_{+}}(v_i)$
 are written as
\begin{align}\label{d:PT}
\delta_{t_{+}}(v_1)=
\delta_{\rm{PT}}(v_1)
&\cneq \left[ 
\left[P_{-n_1}(X, \beta_1)/\mathbb{C}^{\ast} \right] \to 
\mM \right] \\
\label{d:DT}
\delta_{t_{+}}(v_2)= \delta_{\rm{DT}}(v_2)
&\cneq \left[ 
\left[I_{n_2}(X, \beta_2)/\mathbb{C}^{\ast} \right] \to 
\mM \right]
\end{align}
where $\mathbb{C}^{\ast}$ acts on 
$P_{-n_1}(X, \beta_1)$, $I_{n_2}(X, \beta_2)$
trivially. 
The morphisms in (\ref{d:PT}), (\ref{d:DT})
are given by 
sending a stable pair $(\oO_X \to F)$, 
an ideal sheaf $I_C$ to 
the objects
\begin{align}\label{objects}
\oO_X(m_1 H) \otimes \mathbb{D}(\oO_X \to F)[1], \quad
\oO_X(m_2 H) \otimes I_C
\end{align}
respectively. 
Furthermore $(\beta_i, n_i)$ are elements in 
$C(m, \eta H^3/2)$ by Proposition~\ref{prop:extreme}. 
Conversely if 
we are
 given numerical classes (\ref{v12})
satisfying (\ref{revise:1}), 
then Lemma~\ref{lem:revise:1} (i)
 implies that the objects (\ref{objects})
are objects in $\bB_{B, H}$, 
hence the elements $\delta_{\rm{PT}}(v_1)$, 
$\delta_{\rm{DT}}(v_2)$ are well-defined
as above.  

Therefore under the above situation, 
the formula (\ref{delta:sum})
is written as 
\begin{align*}
\delta_{t_0}(v)=\delta_{t_{+}}(v)
+\sum_{\begin{subarray}{c}
v_1=-e^{m_1 H}(1, 0, -\beta_1, -n_1) \\
v_2=e^{m_2 H}(1, 0, -\beta_2, -n_2) \\
m_i \in \mathbb{Z}, \ 
(\beta_i, n_i) \in C(m, \eta H^3/2) \\
\nu_{t_0}(v_1)=\nu_{t_0}(v_2)=0, \ 
v_1 + v_2=v
\end{subarray}}
\delta_{\rm{PT}}(v_1) \ast \delta_{\rm{DT}}(v_2).
\end{align*}
The sum in
the RHS is a finite sum
by Lemma~\ref{lem:revise:1} (ii). 
A similar argument for the 
 $\nu_{t_{-}}$-stability with $t_{-}=t_{0}-\varepsilon$
for $0<\varepsilon \ll 1$ implies that 
\begin{align*}
\delta_{t_0}(v)=\delta_{t_{-}}(v)
+\sum_{\begin{subarray}{c}
v_1=-e^{m_1 H}(1, 0, -\beta_1, -n_1) \\
v_2=e^{m_2 H}(1, 0, -\beta_2, -n_2) \\
m_i \in \mathbb{Z}, \ 
(\beta_i, n_i) \in C(m, \eta H^3/2) \\
\nu_{t_0}(v_1)=\nu_{t_0}(v_2)=0, \ 
v_1 + v_2=v
\end{subarray}}
\delta_{\rm{DT}}(v_2) \ast \delta_{\rm{PT}}(v_1).
\end{align*}
By taking the difference, we obtain 
\begin{align}\label{wall:hall}
\delta_{t_{+}}(v)-\delta_{t_{-}}(v)
=\sum_{\begin{subarray}{c}
v_1=-e^{m_1 H}(1, 0, -\beta_1, -n_1) \\
v_2=e^{m_2 H}(1, 0, -\beta_2, -n_2) \\
m_i \in \mathbb{Z}, \ 
(\beta_i, n_i) \in C(m, \eta H^3/2) \\
\nu_{t_0}(v_1)=\nu_{t_0}(v_2)=0, \ 
v_1 + v_2=v
\end{subarray}}
[\delta_{\rm{DT}}(v_2), \delta_{\rm{PT}}(v_1)].  
\end{align}
Note that any torsion sheaf is an object in $\bB_{B, H}$. 
Hence the moduli stacks $\mM^{ss}_H(v)$, $\mM^{s}_H(v)$
are open substacks of $\mM$. Therefore 
we can define the following elements: 
\begin{align*}
\delta_H(v) &\cneq 
[\mM^{ss}_{H}(v) \subset \mM] \in H(\bB_{B, H}) \\
\delta^{s}_H(v) &\cneq [\mM^{s}_H(v) \subset \mM] \in H(\bB_{B, H}).  
\end{align*}
Now we take 
$\xi$, $\mu$ and $m(\xi, \mu)$
as in Corollary~\ref{cor:mustable}, 
and assume that $m>m(\xi, \mu)$. 
By replacing $m(\xi, \mu)$, we may assume that 
$m>m(\xi, \mu)$ implies 
\begin{align*}
\frac{\mu}{m^{\xi}} \le \mathrm{min}
\left\{ \frac{3}{2m}, \frac{1}{2} \right\}. 
\end{align*}
Then if 
$0\le \eta< \mu/m^{\xi}$, 
applying the formula (\ref{wall:hall})
from $t_0 \gg 0$ to 
$0< t_0 \ll 1$
and using the results of Corollary~\ref{cor:mustable}, 
Lemma~\ref{lem:obvious},
we obtain the following formula: 
\begin{align}\label{deltas}
\delta_H^{s}(v)=
\delta_{H}(v)= \sum_{\begin{subarray}{c}
v_1=-e^{m_1 H}(1, 0, -\beta_1, -n_1) \\
v_2=e^{m_2 H}(1, 0, -\beta_2, -n_2) \\
m_i \in \mathbb{Z}, \ 
(\beta_i, n_i) \in C(m, \eta H^3/2) \\ 
v_1 + v_2=v, \ 
\nu_t(v_i)=0 \ \mathrm{for} \ \mathrm{some} \ t>0 
\end{subarray}}
[\delta_{\rm{DT}}(v_2), \delta_{\rm{PT}}(v_1)].
\end{align}
Let us set $\delta \cneq \mu H^3 /2$. 
Since $0\le \eta < \mu/m^{\xi}$, 
the condition $(\beta_i, n_i) \in C(m, \eta H^3/2)$
can be replaced by $(\beta_i, n_i) \in C(m, \delta/m^{\xi})$, 
 and by replacing $m(\xi, \mu)$ if necessary,  
the condition 
`$\nu_t(v_i)=0$ for some $t>0$'
in (\ref{deltas}) automatically 
follows from the condition
 $(\beta_i, n_i) \in C(m, \delta/m^{\xi})$. 
Therefore the formula (\ref{deltas}) can be modified by 
\begin{align}\label{form:final}
\delta_H^s(v)=
\delta_{H}(v)= \sum_{\begin{subarray}{c}
v_1=-e^{m_1 H}(1, 0, -\beta_1, -n_1) \\
v_2=e^{m_2 H}(1, 0, -\beta_2, -n_2) \\
m_i \in \mathbb{Z}, \ 
(\beta_i, n_i) \in C(m, \delta/m^{\xi}) \\ 
v_1 + v_2=v
\end{subarray}}
[\delta_{\rm{DT}}(v_2), \delta_{\rm{PT}}(v_1)].
\end{align}
The formula (\ref{form:final}) 
is interpreted as a relationship in a certain 
Lie subalgebra $H^{\rm{Lie}}(\bB_{B, H}) \subset H(\bB_{B, H})$, 
the Lie algebra of virtual indecomposable objects
in~\cite{JS}. 
The result of~\cite[Theorem~5.12]{JS} is that there is a Lie 
algebra homomorphism from $H^{\rm{Lie}}(\Coh(X))$, 
(not $H^{\rm{Lie}}(\bB_{B, H})$,)
to a certain Lie algebra defined by the anti-symmetric 
bilinear form $\chi$ on $H^{\ast}(X, \mathbb{Q})$, 
defined by 
\begin{align*}
&\chi\left((r_1, D_1, \gamma_1, s_1), (r_2, D_2, \gamma_2, s_2)  \right) \\
&=r_1 s_2 -\gamma_2 D_1 + \gamma_1 D_2 -r_2 s_1 + \frac{1}{12}c_2(X)
(r_1 D_2 -r_2 D_1). 
\end{align*}
Note that by the Riemann-Roch theorem and the 
Serre duality, for $E_1, E_2 \in \bB_{B, H}$ we have 
\begin{align*}
\chi(\ch(E_1), \ch(E_2))=&
\dim \Hom(E_1, E_2)-\dim \Ext^1(E_1, E_2) \\
 &+ \dim \Ext^1(E_2, E_1)
-\dim \Hom(E_2, E_1). 
\end{align*}
In order to apply the argument of~\cite[Theorem~5.12]{JS}
to our derived category setting,
we need to know the following local property of the 
 moduli stack of objects in $\bB_{B, H}$. 
Recall that $\mM$ is the moduli stack of objects in $\bB_{B, H}$. 
\begin{conj}\label{conj:CS}
 For
any $[E] \in \mM$, 
let $G$ be a maximal reductive 
subgroup in $\Aut(E)$.
Then there exists a $G$-invariant analytic 
open neighborhood $V$ of $0$ in 
$\Ext^1(E, E)$, 
a $G$-invariant holomorphic function $f\colon V\to \mathbb{C}$
with $f(0)=df|_{0}=0$, and a smooth morphism 
of complex analytic stacks
\begin{align*}
\Phi \colon [\{df=0\}/G] \to \mM,
\end{align*}
of relative dimension $\dim \Aut(E)- \dim G$. 
\end{conj}
The above conjecture 
is a derived category version of~\cite[Theorem~5.3]{JS}
and proved if $E \in \Coh(X)$ in~\cite[Theorem~5.3]{JS}.
Also a similar result is already announced by Behrend-Getzler~\cite{BG}. 

If we also assume Conjecture~\ref{conj:CS}, 
then we have a
derived category version of~\cite[Theorem~5.12]{JS}. 
Namely let $\Gamma \subset H^{\ast}(X, \mathbb{Q})$
be the finitely generated free abelian group 
defined by 
\begin{align*}
\Gamma \cneq \Imm \left( \ch \colon K(X) 
\to H^{\ast}(X, \mathbb{Q})\right). 
\end{align*} 
The Lie algebra $C(\Gamma)$ is defined by 
\begin{align*}
C(\Gamma) \cneq \bigoplus_{w\in \Gamma}
\mathbb{Q} c_{w}, 
\end{align*}
with Lie bracket given by 
\begin{align*}
[c_{w_1}, c_{w_2}] \cneq (-1)^{\chi(w_1, w_2)}
\chi(w_1, w_2) c_{w_1 +w_2}. 
\end{align*}
Then there is a
 Lie algebra homomorphism 
\begin{align*}
\Upsilon \colon H^{\rm{Lie}}(\bB_{B, H}) 
\to C(\Gamma), 
\end{align*}
sending 
$\delta_H^s(v)=\delta_{H}(v)$, $\delta_{\rm{PT}}(v_1)$
and $\delta_{\rm{DT}}(v_2)$ to 
\begin{align*}
-\DT_{H}(v) c_v, \ 
-P_{-n_1, \beta_1}c_{v_1}, \ 
-I_{n_2, \beta_2}c_{v_2}, 
\end{align*}
respectively. 
Applying $\Upsilon$ to 
the formula (\ref{form:final}), 
we obtain the following result: 
\begin{thm}\label{thm:DTformula}
Suppose that Conjecture~\ref{conj:BMT} and Conjecture~\ref{conj:CS}
are true. 
Take $v=(0, mH, -\beta, -n) \in H^{\ast}(X, \mathbb{Q})$
with $m\in \mathbb{Z}_{\ge 1}$ and define $\eta$ as in (\ref{def:eta}). 
Also take $(\xi, \mu)$ so that 
$\xi>1, \mu>0$ or $\xi=1$, $0< \mu < 3/2$. 
Then there is $m(\xi, \mu)>0$, which depends only 
on $\xi$ and $\mu$ such that if $m>m(\xi, \mu)$
and $0\le \eta< \mu/m^{\xi}$,
by setting $\delta=\mu H^3/2$, 
 we have 
\begin{align}\label{DT:wcross}
\DT_H(v)= \sum_{\begin{subarray}{c}
v_1=-e^{m_1 H}(1, 0, -\beta_1, -n_1) \\
v_2=e^{m_2 H}(1, 0, -\beta_2, -n_2) \\
m_i \in \mathbb{Z}, \ 
(\beta_i, n_i) \in C(m, \delta/m^{\xi}), \\ 
v_1 + v_2=v
\end{subarray}}
(-1)^{\chi(v_2, v_1)-1}
\chi(v_2, v_1) I_{n_2, \beta_2} P_{-n_1, \beta_1}. 
\end{align}
\end{thm}

\subsection{The formula for the generating series}
The formula (\ref{DT:wcross})
gives a relationship among invariants
(\ref{DT:0m}), (\ref{inv:I})
and (\ref{inv:P}). 
We finally rearrange the formula (\ref{DT:wcross})
in terms of generating series, and 
prove Theorem~\ref{intro:main}. 

For $\epsilon>0$, we define 
the series 
$\zZ_{\rm{D}6-\overline{\rm{D}6}}^{m, \epsilon}(x, y, z)$
to be
\begin{align}\notag
\zZ_{\rm{D}6-\overline{\rm{D}6}}^{m, \epsilon}(x, y, z) \cneq 
&\sum_{\begin{subarray}{c}m_2-m_1=m \\
m_1, m_2 \in \mathbb{Z}
\end{subarray}}
x^{\frac{H^3}{6}(m_1^3-m_2^3)}
y^{\frac{H^2}{2}(m_1^2-m_2^2)}
z^{\frac{H^3}{6}m^3 + \frac{c_2(X)H}{12}m} \\
&\label{D6antiD6} \quad I^{m, \epsilon}(xz^{-1}, x^{m_2 H}y z^{-mH})
P^{m, \epsilon}(x z^{-1}, x^{-m_1 H}y^{-1} z^{-mH}). 
\end{align}
Although the above sum is an infinite sum, 
the coefficient of each term is a finite sum
and the series 
$\zZ_{\rm{D}6-\overline{\rm{D}6}}^{m, \epsilon}(x, y, z)$
is well-defined. 

We see a relationship between 
$\zZ_{\rm{D}4}^{m}(x, y)$ defined by (\ref{ZD4})
and the series 
$\zZ_{\rm{D}6-\overline{\rm{D}6}}^{m, \epsilon}(x, y, z)$. 
By expanding, $\zZ_{\rm{D}6-\overline{\rm{D}6}}^{m, \epsilon}(x, y, z)$
is written as 
\begin{align*}
\sum_{\begin{subarray}{c}
m_i \in \mathbb{Z}, \ (\beta_i, n_i) \in C(m, \epsilon) \\
i=1, 2, \ m_2-m_1=m
\end{subarray}}
&I_{n_2, \beta_2}P_{-n_1, \beta_1} 
x^{n_2-n_1+m_2 \beta_2 H -m_1 \beta_1 H + \frac{H^3}{6}(m_1^3 -m_2^3)} \\ 
&y^{\beta_2 -\beta_1 + \frac{H^2}{2}(m_1^2 -m_2^2)}
z^{n_1 -n_2 -m(\beta_1 +\beta_2)H
+ \frac{H^3}{6}m^3 + \frac{c_2(X)H}{12}m}.
\end{align*}
On the other hand, in the formula (\ref{DT:wcross}), the 
condition $v_1 + v_2=v$ is 
equivalent to the
following conditions: 
\begin{align*}
&m_2 -m_1 =m, \\
&\beta_2 -\beta_1 + \frac{H^2}{2}(m_1^2 -m_2^2)=\beta, \\
&n_2-n_1+m_2 \beta_2 H -m_1 \beta_1 H + \frac{H^3}{6}(m_1^3 -m_2^3)
=n. 
\end{align*}
Also the Euler pairing $\chi(v_2, v_1)$ is 
computed by 
\begin{align*}
\chi(v_2, v_1)= n_1 -n_2 -m(\beta_1 +\beta_2)H
+ \frac{H^3}{6}m^3 + \frac{c_2(X)H}{12}m. 
\end{align*}
Therefore if we take $\xi$, $\mu$
and $m>m(\mu, \xi)$ as in 
Theorem~\ref{thm:DTformula}, 
the formula (\ref{DT:wcross}) implies the equality, 
\begin{align*}
\zZ_{\rm{D}4}^{m}(x, y)
=\left. \frac{\partial}{\partial z} \zZ_{\rm{D}6-\overline{\rm{D}6}}^{m, \epsilon=\frac{\delta}{m^{\xi}}}(x, y, z) \right|_{z=-1},
\end{align*}
for the terms $x^{n}y^{\beta}$ such that 
$(m, \beta, n)$ satisfies 
$\eta<\mu/m^{\xi}$, where $\eta$ is defined by (\ref{def:eta}). 
The last condition is equivalent to 
\begin{align*}
 -\frac{H^3}{24}m^3 
\left( 1-\frac{\mu}{m^{\xi}}  \right) >
n+ \frac{(\beta \cdot H)^2}{2mH^3}. 
\end{align*}
As a summary, we obtain the 
following result which proves Theorem~\ref{intro:main}: 
\begin{thm}\label{thm:goal}
Let $X$ be a smooth projective Calabi-Yau 3-fold
such that $\Pic(X)$ is generated by $\oO_X(H)$ for 
an ample divisor $H$ in $X$. 
Suppose that Conjecture~\ref{conj:BMT} and 
Conjecture~\ref{conj:CS}
are true. Then for any $(\xi, \mu) \in \mathbb{R}^2$
satisfying $\xi>1, \mu>0$ or $\xi=1$, $0<\mu<3/2$, 
there is $m(\xi, \mu)>0$ which depends only 
on $\xi$, $\mu$ such that if $m>m(\xi, \mu)$, 
by setting $\delta=\mu H^3/2$, 
there is an equality of the generating 
series, 
\begin{align}\label{form:goal}
\zZ_{\rm{D}4}^{m}(x, y)
=\left. \frac{\partial}{\partial z} \zZ_{\rm{D}6-\overline{\rm{D}6}}^{m, \epsilon=\frac{\delta}{m^{\xi}}}(x, y, z) \right|_{z=-1},
\end{align}
modulo terms of $x^n y^{\beta}$ with 
\begin{align}\notag
-\frac{H^3}{24}m^3
 \left(1-\frac{\mu}{m^{\xi}} \right)
\le n +\frac{(\beta \cdot H)^2}{2mH^3}.
\end{align}
\end{thm}
\begin{rmk}
Even when $\Pic(X)$ is not generated by 
an ample line bundle $\oO_X(H)$, many of the arguments
in this section work. However in this case, 
there is a difficulty in 
evaluating numerical classes of curves, 
which corresponds to Proposition~\ref{prop:extreme}. 
Unfortunately at this moment, we are 
not able to generalize the arguments 
without the condition $\Pic(X)=\mathbb{Z}[\oO_X(H)]$. 
\end{rmk}

\section{Appendix}\label{sec:append}
\subsection{Examples of wall-crossing}\label{subsec:Exam}
We investigate some examples
of wall-crossing phenomena 
in the previous subsection. 
Let us take a smooth member
$P \in \lvert mH \rvert$ 
and a zero dimensional subscheme 
$Z \subset X$ such that $\oO_Z$
has length $N$. 
We consider the object, 
\begin{align}\label{IPZ}
E=i_{\ast}I_{P, Z} \in \Coh_{\le 2}(X), 
\end{align}
where $i\colon P \hookrightarrow X$ is the inclusion
and $I_{P, Z} \subset \oO_P$ is the defining 
ideal of $Z$ in $P$. 
Note that we have 
\begin{align*}
\ch(E)=\left(0, mH, -\frac{H^2}{2}m^2, \frac{H^3}{6}m^3 -N \right).
\end{align*}
Hence we take $B=-mH/2$ and $\ch^{B}(E)$ is 
written as (\ref{calcu}) 
with $\eta$ given by 
\begin{align*}
\eta=\frac{24N}{m^3 H^3}. 
\end{align*}
We assume that $\eta<\mu/m^{\xi}$ and $m>m(\xi, \mu)$
as in Theorem~\ref{thm:goal}.
The object $E$ is $\mu_{H, 2}$-stable, 
hence $\nu_t$-semistable for $t\ge \sqrt{3}m/2$
by Proposition~\ref{prop:larget}. 

Let us look at the $\nu_t$-stability of $E$
for $t<\sqrt{3}m/2$. 
If we denote by $I_{X, Z} \subset \oO_X$
the defining ideal sheaf of $Z$ in $X$, 
we have the distinguished triangle, 
\begin{align}\label{IEO}
I_{X, Z} \to E \to \oO_X(-mH)[1]. 
\end{align}
The above triangle is an exact sequence in 
$\bB_{B, H}$. Moreover we have 
\begin{align*}
\nu_{t_{-}}(I_{X, Z})> \nu_{t_{-}}(\oO_X(-mH)[1]),
\end{align*}
where $t_{-}=\left( \sqrt{3}m/2 \right) -\varepsilon$
for $0< \varepsilon \ll 1$. 
Here we have observed the wall-crossing phenomena: 
the object $E$ is no longer $\nu_t$-semistable 
for $t=t_{-}$. Instead one
might try to flip the sequence (\ref{IEO})
and consider a sequence, 
\begin{align*}
\oO_X(-mH)[1] \to E' \to I_{X, Z}. 
\end{align*}
If
 the above sequence does not split,
then the 
object $E'$
is $\nu_{t_{-}}$-semistable and 
 coincides with an object
considered in~\cite{BBMT} 
up to tensoring a line bundle. 

If we assume Conjecture~\ref{conj:BMT}, 
then $E'$ is not $\nu_t$-semistable 
for $0<t\ll 1$ by Lemma~\ref{lem:tsmall}. 
Hence there should exist
$t' <\sqrt{3}m/2$
such that $E'$ is $\nu_t$-semistable 
for $t\in [t', \sqrt{3}m/2]$
but not $\nu_{t'_{-}}$-semistable 
where $t_{-}'=t'-\varepsilon$
for $0<\varepsilon \ll 1$. 
An argument of~\cite[Proposition~3.3]{BMT} shows 
that a destabilizing sequence of $E'$
with respect to $\nu_{t'_{-}}$-stability 
should be of the following form, 
\begin{align}\label{IE'E}
I_{C} \to E' \to E'',
\end{align}
where $C \subset X$ is a curve in $X$
which contains $Z$. The object $E''$ is of the form
\begin{align*}
E'' \cong \oO_X(-mH) \otimes \mathbb{D}(\oO_X \to F)[1], 
\end{align*}
where $F$ is a pure sheaf supported on $C$
and $\oO_X \to F$ is a PT stable pair. 
In this way, we observe that 
curves in $X$ appear starting from 
the object (\ref{IPZ}). 

By Lemma~\ref{lem:tsmall},
we must have $t'\ge \sqrt{3}m/2 \cdot \sqrt{1-\eta}$. 
Since $\nu_{t'}(I_C)=0$, this condition is equivalent to 
\begin{align}\label{HCN}
mH \cdot C \le 3N. 
\end{align}
Indeed if $N=1$, we can  
find such a curve $C$ without assuming 
Conjecture~\ref{conj:BMT}. 
In turn, the
existence of the curve $C$
can be used to show Conjecture~\ref{conj:BMT} 
for the object $E'$. 
(cf.~\cite[Proposition~4.4]{BBMT}.)
As we explained in~\cite[Proposition~4.4]{BBMT}, 
the curve $C$ is found in 
the proof of Fujita's 
freeness conjecture 
on 3-folds~\cite{EL}, \cite{Kaw97}, \cite{Hel97}, 
and is defined by a
multiplier ideal sheaf of some 
log canonical $\mathbb{Q}$-divisor in $X$. 
There is an 
embedding $I_C \hookrightarrow E'$ 
since the composition 
$I_C \to I_{X, Z} \to \oO_X(-mH)[2]$
vanishes by Nadel's vanishing theorem. 
\begin{rmk}
When $N >1$, we can at least 
find a curve $C \subset X$ 
satisfying (\ref{HCN})
and 
$Z \cap C \neq \emptyset$,
without assuming Conjecture~\ref{conj:BMT}, 
following the proof of~\cite[Theorem~6.1]{Hel97}.
Unfortunately the argument of~\cite[Theorem~6.1]{Hel97}
is not enough to find such a curve $C$
with $Z\subset C$, which is necessary
to solve Conjecture~\ref{conj:BMT}
for the object $E'$. 
\end{rmk}

\subsection{Euler characteristic version}
If we do not assume the inequality in 
 Conjecture~\ref{conj:CS}
and just assume Conjecture~\ref{conj:BMT}, 
we still have a result for 
Euler characteristic invariants, 
which are not weighted by the Behrend function. 
By formally putting $\nu \equiv 1$
in the definitions of (\ref{DT:0m}), (\ref{inv:I})
and (\ref{inv:P}),
where $\nu$ is the Behrend function, we 
can define the Euler characteristic invariants, 
\begin{align}\label{inv:Eu}
\widehat{\DT}_{H}(0, mH, -\beta, -n),  \
\widehat{I}_{n, \beta}, \
\widehat{P}_{n, \beta}. 
\end{align}
By replacing (\ref{DT:0m}), (\ref{inv:I}), (\ref{inv:P}) 
in the generating series (\ref{ZD4}), (\ref{D6antiD6})
by the invariants (\ref{inv:Eu})
respectively, we can define the generating series,
\begin{align*}
\widehat{\zZ}_{\rm{D}4}^{m}(x, y), \quad 
\widehat{\zZ}_{\rm{D}6-\overline{\rm{D}6}}^{m, \epsilon}(x, y, z).
\end{align*}
The following result can be proved 
along with the same proof of Corollary~\ref{cor:conji} and 
Theorem~\ref{thm:goal}. 
The only modification is that we 
use the result of~\cite[Theorem~6.12]{Joy2} 
instead of~\cite[Theorem~5.12]{JS}. 
\begin{thm}\label{thm:goal:euler}
Let $X$ be a smooth projective Calabi-Yau 3-fold
such that $\Pic(X)$ is generated by $\oO_X(H)$ for 
an ample divisor $H$ in $X$. 
Suppose that Conjecture~\ref{conj:BMT}
is true. 
Then we have the following:

(i) The invariant   
$\widehat{\DT}_H(0, mH, -\beta, -n)$
vanishes 
unless 
\begin{align*}
-\frac{H^3}{24}m^3 \le n+ \frac{(\beta \cdot H)^2}{2mH^3}. 
\end{align*}

(ii) For any $(\xi, \mu) \in \mathbb{R}^2$
satisfying $\xi>1, \mu>0$ or $\xi=1$, $0<\mu<3/2$, 
there is $m(\xi, \mu)>0$ which depends only 
on $\xi$, $\mu$ such that if $m>m(\xi, \mu)$, 
by setting $\delta=\mu H^3/2$, 
there is an equality of the generating 
series, 
\begin{align}\notag
\widehat{\zZ}_{\rm{D}4}^{m}(x, y)
=\left. \frac{\partial}{\partial z} \widehat{\zZ}_{\rm{D}6-\overline{\rm{D}6}}^{m, \epsilon=\frac{\delta}{m^{\xi}}}(x, y, z) \right|_{z=1},
\end{align}
modulo terms of $x^n y^{\beta}$ with 
\begin{align}\notag
-\frac{H^3}{24}m^3
 \left(1-\frac{\mu}{m^{\xi}} \right)
\le n +\frac{(\beta \cdot H)^2}{2mH^3}.
\end{align}
\end{thm}

\subsection{Evidence of Conjecture~\ref{intro:conj} (i)}
In the situation of Corollary~\ref{cor:conji}, the inequality (\ref{ineq:mH})
can be also written as 
\begin{align}\label{ineq:ch:tor}
\ch_3(E) \le \frac{1}{24} \ch_1(E)^3 + 
\frac{(\ch_1(E) \cdot \ch_2(E))^2}{2(\ch_1(E))^3}. 
\end{align}
Now let $i \colon S \hookrightarrow X$ be a smooth ample divisor
and $E$ is supported on $S$, i.e. 
\begin{align*}
E \in \Coh_S(X) \cneq \{ F \in \Coh(X) : 
\Supp(F) \subset S\}. 
\end{align*}
Note that $E$ may not be an $\oO_S$-module, 
but an $\oO_{S'}$-module for some thickening 
$S \subset S'$. 

The inclusion $\Coh(S) \subset \Coh_S(X)$
induces the isomorphism of $K$-groups
\begin{align*}
K(\Coh(S)) \stackrel{\cong}{\to} K(\Coh_S(X)),
\end{align*} hence 
the class $[E] \in K(\Coh_S(X))$ is 
regarded as an element $[E] \in K(\Coh(S))$. 
Taking its Chern character, we can define 
\begin{align*}
\ch^{S}([E]) =(r, l, s) \in H^0(S) \oplus H^2(S) \oplus H^4(S). 
\end{align*}
It is related to the usual $\ch(E) \in H^{\ast}(X, \mathbb{Q})$
by 
\begin{align}\label{ch:chS}
\ch(E)=\left(0, rS, -\frac{r}{2}S^2 +i_{\ast}l, 
\frac{r}{6}S^3 -\frac{1}{2}S \cdot i_{\ast}l+s \right),
\end{align}
by the Grothendieck Riemann-Roch theorem. Substituting (\ref{ch:chS})
 to (\ref{ineq:ch:tor}),
and writing $\lL \cneq \oO_X(S)|_{S} \in \Pic(S)$,  
we obtain 
\begin{align}\label{ineq:derive}
s \le \frac{\lL^2}{24}(r^3 -r) + \frac{\lL \cdot l}{2r \lL^2}. 
\end{align} 
The above inequality is derived by assuming Conjecture~\ref{conj:BMT}. 
When the formal neighborhood 
$S \subset X$ is isomorphic to 
$S \subset \lvert \lL \rvert$, where 
the latter embedding is the zero section, then 
we can show the inequality (\ref{ineq:derive}), 
(or rather a stronger one,) directly
without assuming Conjecture~\ref{conj:BMT}. 
The following theorem gives an evidence of Conjecture~\ref{intro:conj} (i), 
hence Conjecture~\ref{conj:BMT}:
(note that the inequality (\ref{ineq:derive2}) below
implies (\ref{ineq:derive}) by the Hodge index theorem.)
\begin{thm}\label{thm:append}
Let $S$ be a smooth projective surface over $\mathbb{C}$
and $\lL$ an ample line bundle on $S$. 
Let $Y$ be the total space of the line bundle $\lL$, 
$\pi \colon Y \to S$ the projection, 
and $S$ is considered as a subvariety of $Y$ via 
zero section. Then for any $\pi^{\ast}\lL$-semistable 
pure two dimensional sheaf $E \in \Coh_{S}(Y)$, 
we have the inequality, 
\begin{align}\label{ineq:derive2}
s \le \frac{\lL^2}{24}(r^3 -r) + \frac{l^2}{2r},
\end{align}
where $(r, l, s) =\ch(\pi_{\ast}E)=\ch^{S}([E]) \in H^{\ast}(S, \mathbb{Q})$. 
\end{thm}
\begin{proof}
Let us consider the object $\pi_{\ast}E \in \Coh(S)$ 
and its Harder-Narasimhan filtration 
with respect to $\mu_{\lL}$-stability, 
\begin{align*}
0= E_0 \subset E_1 \subset \cdots \subset E_N=\pi_{\ast}E.
\end{align*}
Here $F_i=E_i/E_{i-1}$ is 
$\mu_{\lL}$-semistable with $\mu_{\lL}(F_i)>\mu_{\lL}(F_{i+1})$
for all $i$. 
Note that the $\oO_Y$-module structure on 
$E$ induces the morphism, 
\begin{align*}
\theta \colon \pi_{\ast}E \to \pi_{\ast}E\otimes \lL. 
\end{align*}
The $\pi^{\ast} \lL$-semistability of $E$
implies that the subsheaf 
$E_i \subset \pi_{\ast}E$ is not preserved by $\theta$. 
Therefore the composition 
\begin{align*}
E_i \to \pi_{\ast}E \stackrel{\theta}{\to} \pi_{\ast}E \otimes \lL 
\to (\pi_{\ast}E/E_i) \otimes \lL,
\end{align*}
is non-zero, which implies 
\begin{align}\label{ineq:mu}
\mu_{\lL}(F_i) \le \mu_{\lL}(F_{i+1} \otimes \lL)
= \mu_{\lL}(F_{i+1}) + \lL^2. 
\end{align}
If we write $\ch(F_i)=(r_i, l_i, s_i)$, 
then the inequality (\ref{ineq:mu}) implies  
\begin{align}\label{ineq:ij}
(j-i) \lL^2 \ge \frac{r_j l_i - r_i l_j}{r_i r_j} \cdot \lL >0,
\end{align}
for any $j>i$.
On the other hand, we have 
$l_i^2 \ge 2r_i s_i$ by Bogomolov-Gieseker inequality.
Therefore we have  
\begin{align}\label{desire1}
s-\frac{l^2}{2r}
\le \sum_{i=1}^{N} \frac{l_i^2}{2r_i} 
- \frac{\left(\sum_{i=1}^{N}l_i\right)^2}{\sum_{i=1}^{N}2r_i}
=\sum_{1\le i<j \le N}
\frac{(r_j l_i -r_i l_j)^2}{2r_i r_j r}. 
\end{align}
By the Hodge index theorem and (\ref{ineq:ij}), 
we have 
\begin{align*}
(r_j l_i -r_i l_j)^2 \le \{ (r_j l_i -r_i l_j) \cdot \lL \}^2 /\lL^2  
\le r_i^2 r_j^2 (j-i)^2 \lL^2. 
\end{align*}
Hence we have
\begin{align}\label{desire2}
\sum_{1\le i<j \le N}
\frac{(r_j l_i -r_i l_j)^2}{2r_i r_j r}
\le \frac{\lL^2}{2r} \sum_{1\le i< j\le N}
r_i r_j(j-i)^2. 
\end{align}
The desired inequality (\ref{ineq:derive2})
follows from (\ref{desire1}), (\ref{desire2})
and noting that
\begin{align*}
\sum_{1\le i<j \le N} r_i r_j(j-i)^2 
\le \sum_{1\le i< j\le r} (j-i)^2
= \frac{1}{12}r^2(r^2-1). 
\end{align*}
\end{proof}
\begin{rmk}
The inequality (\ref{ineq:derive2}) 
is best possible. In fact the equality 
is achieved
when $E$ is a structure sheaf of a 
thickened surface $\oO_{S_m} \cneq
\oO_Y/\oO_Y(-mS)$
for  
$m\in \mathbb{Z}_{\ge 1}$.  
\end{rmk}

Todai Institute for Advanced Studies (TODIAS), 

Kavli Institute for the Physics and Matheamtics of 
the Universe, 

University of Tokyo, 
5-1-5 Kashiwanoha, Kashiwa, 277-8583, Japan.

\textit{E-mail address}: yukinobu.toda@ipmu.jp


\begin{thebibliography}{10}

\bibitem{AM}
E.~Andriyash and G.~Moore.
\newblock Ample {D}4-{D}2-{D}0 decay.
\newblock arXiv:hep-th/0806.4960.

\bibitem{BBMT}
A.~Bayer, A.~Bertram~E. Macri, and Y.~Toda.
\newblock Bridgeland stability conditions on 3-folds {II}: {A}n application to
  {F}ujita's conjecture.
\newblock {\em preprint}.
\newblock arXiv:1106.3430.

\bibitem{BMT}
A.~Bayer, E.~Macri, and Y.~Toda.
\newblock Bridgeland stability conditions on 3-folds {I}:
  {B}ogomolov-{G}ieseker type inequalities.
\newblock {\em J.~Algebraic Geom.~(to appear)}.
\newblock arXiv:1103.5010.

\bibitem{Beh}
K.~Behrend.
\newblock Donaldson-{T}homas invariants via microlocal geometry.
\newblock {\em Ann.~of Math}, Vol. 170, pp. 1307--1338, 2009.

\bibitem{BG}
K.~Behrend and E.~Getzler.
\newblock Chern-{S}imons functional.
\newblock {\em in preparation}.

\bibitem{Brs1}
T.~Bridgeland.
\newblock Stability conditions on triangulated categories.
\newblock {\em Ann.~of Math}, Vol. 166, pp. 317--345, 2007.

\bibitem{BrH}
T.~Bridgeland.
\newblock Hall algebras and curve-counting invariants.
\newblock {\em J.~Amer.~Math.~Soc.~}, Vol.~24, pp. 969--998, 2011.

\bibitem{DM}
F.~Denef and G.~Moore.
\newblock Split states, {E}ntropy {E}nigmas, {H}oles and {H}alos.
\newblock arXiv:hep-th/0702146.

\bibitem{DM07}
E.~Diaconescu and G.~Moore.
\newblock Crossing the {W}all: {B}ranes vs. {B}undles.
\newblock arXiv:hep-th/0706.3193.

\bibitem{EL}
L.~Ein and R.~Lazarsfeld.
\newblock Global generation of pluri canonical and adjoint linear series on
  smooth projective threefolds.
\newblock {\em J.~Amer.~Math.~Soc.~}, Vol.~6, pp. 875--903, 1993.

\bibitem{Fujita}
T.~Fujita.
\newblock On polarized manifolds whose adjoint bundles are not semipositive.
\newblock {\em Adv.~Stud.~Pure Math.~}, Vol.~10, pp. 167--178, 1987.
\newblock Algebraic Geometry, Sendai, 1985.

\bibitem{GGHS05}
D.~Gaiotto, M~Guica, L.~Huang, A.~Simons, A.~Strominger, and X.~Yin.
\newblock {D}4-{D}0 branes on the quintic.
\newblock arXiv:hep-th/0509168.

\bibitem{GSY06}
D.~Gaiotto, A.~Strominger, and X.~Yin.
\newblock The {M}5-brane elliptic genus: {M}odularity and {BPS} states.
\newblock arXiv:hep-th/0607010.

\bibitem{Got}
L.~G$\ddot{\rm{o}}$ttsche.
\newblock The {B}etti numbers of the {H}ilbert scheme of points on a smooth
  projective surface.
\newblock {\em Math.~Ann.~}, Vol. 286, pp. 193--207, 1990.

\bibitem{HRS}
D.~Happel, I.~Reiten, and S.~O. Smal$\o$.
\newblock {\em Tilting in abelian categories and quasitilted algebras}, Vol.
  120 of {\em Mem.~Amer.~Math.~Soc}.
\newblock 1996.


\bibitem{Hel97}
S.~Helmke.
\newblock On {F}ujita's conjecture.
\newblock {\em Duke Math.~J.~}, Vol.~88, pp. 201--216, 1997.



\bibitem{Hu}
D.~Huybrechts and M.~Lehn.
\newblock {\em Geometry of moduli spaces of sheaves}, Vol. E31 of {\em Aspects
  in Mathematics}.
\newblock Vieweg, 1997.

\bibitem{Inaba}
M.~Inaba.
\newblock Toward a definition of moduli of complexes of coherent sheaves on a
  projective scheme.
\newblock {\em J.~Math.~Kyoto Univ.~}, Vol. 42-2, pp. 317--329, 2002.

\bibitem{Joy2}
D.~Joyce.
\newblock Configurations in abelian categories {I}\hspace{-.1em}{I}.
  {R}ingel-{H}all algebras.
\newblock {\em Advances in Math}, Vol. 210, pp. 635--706, 2007.

\bibitem{Joy4}
D.~Joyce.
\newblock Configurations in abelian categories {I}\hspace{-.1em}{V}.
  {I}nvariants and changing stability conditions.
\newblock {\em Advances in Math}, Vol. 217, pp. 125--204, 2008.

\bibitem{JS}
D.~Joyce and Y.~Song.
\newblock A theory of generalized {D}onaldson-{T}homas invariants.
\newblock {\em Memoirs of the A.~M.~S.~(to appear)}.
\newblock arXiv:0810.5645.

\bibitem{Kaw97}
Y.~Kawamata.
\newblock On {F}ujita's freeness conjecture for 3-folds and 4-folds.
\newblock {\em Math.~Ann.~}, Vol. 308, pp. 491--505, 1997.

\bibitem{K-S}
M.~Kontsevich and Y.~Soibelman.
\newblock Stability structures, motivic {D}onaldson-{T}homas invariants and
  cluster transformations.
\newblock {\em preprint}.
\newblock arXiv:0811.2435.

\bibitem{Langer}
A.~Langer.
\newblock Semistable sheaves in positive characteristic.
\newblock {\em Ann.~of Math.~}, Vol. 159, pp. 251--276, 2004.

\bibitem{Langer2}
A.~Langer.
\newblock Moduli spaces of sheaves and principal {$G$}-bundles.
\newblock {\em Proc.~Sympos.~Pure Math.~}, Vol.~80, pp. 273--308, 2009.

\bibitem{LIE}
M.~Lieblich.
\newblock Moduli of complexes on a proper morphism.
\newblock {\em J.~Algebraic Geom.~}, Vol.~15, pp. 175--206, 2006.

\bibitem{MNOP}
D.~Maulik, N.~Nekrasov, A.~Okounkov, and R.~Pandharipande.
\newblock Gromov-{W}itten theory and {D}onaldson-{T}homas theory. {I}.
\newblock {\em Compositio.~Math}, Vol. 142, pp. 1263--1285, 2006.

\bibitem{OSV}
H.~Ooguri, A.~Strominger, and C.~Vafa.
\newblock Black hole attractors and the topological string.
\newblock {\em Phys.~Rev.~D}, Vol.~70, , 2004.
\newblock arXiv:hep-th/0405146.

\bibitem{PT}
R.~Pandharipande and R.~P. Thomas.
\newblock Curve counting via stable pairs in the derived category.
\newblock {\em Invent.~Math.~}, Vol. 178, pp. 407--447, 2009.

\bibitem{StTh}
J.~Stoppa and R.~P. Thomas.
\newblock Hilbert schemes and stable pairs: {GIT} and derived category wall
  crossings.
\newblock {\em Bull.~Soc.~Math.~France}, Vol. 139, pp. 297--339, 2011.

\bibitem{SV}
A.~Strominger and C.~Vafa.
\newblock Microscopic origin of the {B}ekenstein-{H}awking entropy.
\newblock {\em Phys.~Lett.~B}, Vol. 379, , 1996.
\newblock arXiv:hep-th/9601029.

\bibitem{Thom}
R.~P. Thomas.
\newblock A holomorphic {C}asson invariant for {C}alabi-{Y}au 3-folds and
  bundles on ${K3}$-fibrations.
\newblock {\em J.~Differential.~Geom}, Vol.~54, pp. 367--438, 2000.

\bibitem{Tst3}
Y.~Toda.
\newblock Moduli stacks and invariants of semistable objects on {K}3 surfaces.
\newblock {\em Advances in Math}, Vol. 217, pp. 2736--2781, 2008.

\bibitem{Tolim}
Y.~Toda.
\newblock Limit stable objects on {C}alabi-{Y}au 3-folds.
\newblock {\em Duke Math.~J.~}, Vol. 149, pp. 157--208, 2009.

\bibitem{Tcurve1}
Y.~Toda.
\newblock Curve counting theories via stable objects~{I}: {DT/PT}
  correspondence.
\newblock {\em J.~Amer.~Math.~Soc.~}, Vol.~23, pp. 1119--1157, 2010.

\bibitem{Tolim2}
Y.~Toda.
\newblock Generating functions of stable pair invariants via wall-crossings in
  derived categories.
\newblock {\em Adv.~Stud.~Pure Math.~}, Vol.~59, pp. 389--434, 2010.
\newblock New developments in algebraic geometry, integrable systems and mirror
  symmetry (RIMS, Kyoto, 2008).

\end{thebibliography}
\end{document}